\documentclass[a4paper]{article}

\usepackage[latin1]{inputenc} 
\usepackage{graphicx}
\usepackage[pdftex,bookmarks]{hyperref}
\usepackage{amsmath}    
\usepackage{amssymb}
\usepackage{amsmath,amsthm}
\usepackage{textcomp}
\usepackage{color}
\usepackage[all]{xy}
\usepackage{geometry}

\newtheorem{teo}{Theorem}[section]
\newtheorem{defi}{Definition}[section]
\newtheorem{lemma}{Lemma}[section]

\newtheorem{cor}{Corollary}[section]
\newtheorem{prop}{Proposition}[section]

\newtheorem{rem} {Remark}[section]

\newcommand{\nn}{\nonumber}

\DeclareMathOperator{\im}{im}

\DeclareMathOperator{\dvol}{dvol}

\DeclareMathOperator{\vol}{vol}
\DeclareMathOperator{\reg}{reg}
\DeclareMathOperator{\sing}{sing}

\DeclareMathOperator{\depth}{depth}

\title{\huge \bf  $L^p$-cohomology, heat semigroup and stratified spaces}  
\author{Francesco Bei  \bigskip \\
Dipartimento di matematica, Sapienza Universit\`a di Roma\\ E-mail addresses: bei@mat.uniroma.it\ \     francescobei27@gmail.com }
\date{}

\begin{document}

\maketitle
\begin{abstract}
Let $(M,g)$ be an incomplete Riemannian manifold of finite volume and let $2\leq p<\infty$. In the first part of this paper we prove that under certain assumptions the inclusion of the space of $L^p$-differential forms into that of $L^2$-differential forms gives rise to an injective/surjective map between the corresponding $L^p$ and $L^2$ cohomology groups. Then in the second part we provide various applications of  these results to the curvature and the intersection cohomology of compact Thom-Mather stratified pseudomanifolds and complex projective varieties with only isolated singularities.
\end{abstract}
\vspace{1 cm}

\vspace{0.5 cm}


\tableofcontents

\section*{Introduction}

Let $X$ be either a compact Thom-Mather stratified pseudomanifold or a singular complex projective variety and let us denote with $g$ either an iterated conic metric on $\reg(X)$ or  the K\"ahler metric  on $\reg(X)$ induced by the Fubini-Study metric, respectively.  Since the seminal papers of Cheeger \cite{JC} and Cheeger-Goresky-MacPherson \cite{CGM} several works have been devoted to show the existence of isomorphisms between the $L^p$ cohomology groups of $\reg(X)$, the regular part of $X$, and some intersection cohomology groups of $X$, establishing in this way a sort of $L^p$-de Rham theorem for an important class singular spaces. Just to mention a small sample of works we can recall here \cite{BGM}, \cite{BHS}, \cite{HS}, \cite{Masa}, \cite{TOh}, \cite{Val} and \cite{BYOU}.  As is well known $\reg(X)$ has finite volume with respect to $g$ and thus there exists a continuous inclusion of Banach spaces $i:L^q\Omega^k(\reg(X),g)\hookrightarrow L^p\Omega^k(\reg(X),g)$ whenever $p<q$. This inclusion is easily seen to induces a morphism of complexes between $(L^q\Omega^*(\reg(X),g),d_{*,\max,q})$ and $(L^p\Omega^*(\reg(X),g),d_{*,\max,p})$, where the former complex is the $L^q$ maximal de Rham complex and the latter is the $L^p$ maximal de Rham complex. Put it differently if $\omega\in \mathcal{D}(d_{*,\max,q})$, the domain of $d_{*,\max,q}:L^q\Omega^{*}(\reg(X),g)\rightarrow L^q\Omega^{*+1}(\reg(X),g)$, then $\omega\in \mathcal{D}(d_{*,\max,p})$ and $i(d_{*,\max,q}\omega)=d_{*,\max,p}i(\omega)$ in $L^p\Omega^{*+1}(\reg(X),g)$. However $\gamma:H^*_{q,\max}(\reg(X),g)\rightarrow H^*_{p,\max}(\reg(X),g)$, the map induced by $i:L^q\Omega^k(\reg(X),g)\hookrightarrow L^p\Omega^k(\reg(X),g)$ between the corresponding cohomology groups, is in general neither injective nor surjective. At this point we can summarize our main goal by saying that we found some answers to the following question:\\

\noindent Under what circumstances is the map $\gamma:H^*_{q,\max}(\reg(X),g)\rightarrow H^*_{p,\max}(\reg(X),g)$ injective or surjective or even better an isomorphism?\\

Let us go now into some more detail by describing the structure of this paper. The first section contains some background material on $L^p$-spaces, differential operators and $L^p$-cohomology. 
In particular  its last subsection is devoted to the definition of the Kato class of a Riemannian manifold, a notion that will play a central role in the rest of this paper. In order to state below our main results let us give a brief account on that: given a possibly incomplete Riemannian manifold $(M,g)$, let $L_k$ denote the curvature term appearing in  the Weitzenb\"ock formula $\Delta_k=\Delta^*\circ \Delta+L_k$. We say that the negative part of $L_k$ lies in the Kato class of $M$, $\ell_k^-\in \mathcal{K}(M)$, if 
\begin{equation}
\label{inkato}
\lim_{t\rightarrow 0^+}\sup_{x\in M}\int_0^t\int_Mp(s,x,y)\ell^-(y)\dvol_g(y)ds=0
\end{equation}
with  $\ell_k^-:M\rightarrow \mathbb{R}$  the function defined as 
$$\ell_k^-(x):=\max\{-l_k(x),0\}\quad\quad \text{and}\quad\quad l_k(x):=\inf_{v\in \Lambda^kT_x^*M,\ g(v,v)=1}g(L_{k,x}v,v).$$ Moreover in \eqref{inkato} $(s,x,y)\in (0,\infty)\times M\times M$, $\Delta_0^{\mathcal{F}}$ is the Friedrich extension of the Laplace-Beltrami operator and $p(s,x,y)$ denotes the smooth kernel of the heat operator $e^{-t\Delta_0^{\mathcal{F}}}:L^2(M,g)\rightarrow L^2(M,g)$. In particular, as we will see later, \eqref{inkato} holds true whenever $L_k\geq c$ for some $c\in \mathbb{R}$.\\ In the second section of this paper we address the above question in the general setting of incomplete Riemannian manifolds of finite volume. Although it looks a natural question it seems to our best knowledge that those of this paper are the first results in the literature that provide some sufficient conditions. Our first main result shows the existence of an injective map in the opposite direction of $\gamma$ between certain reduced maximal $L^p$  cohomology groups:
\begin{teo}
\label{tozzix}
Let $(M,g)$ be an open and incomplete Riemannian manifold of finite volume and dimension $m>2$. Assume in addition that
\begin{itemize}
\item We have a continuous inclusion 
$$W^{1,2}_0(M,g)\hookrightarrow L^{\frac{2m}{m-2}}(M,g)$$
\item There exists $k\in \{0,...,m\}$ such that $\ell_k^-\in \mathcal{K}(M)$;
\item  the operator 
$$\Delta_{k,\mathrm{abs}}:L^2\Omega^k(M,g)\rightarrow L^2\Omega^{k}(M,g)\quad\quad equals\quad\quad  \Delta_{k}^{\mathcal{F}}:L^2\Omega^k(M,g)\rightarrow L^2\Omega^{k}(M,g).$$
\end{itemize}
Then for any $2\leq p\leq \infty$ we have a continuous inclusion 
$$ \mathcal{H}^k_{2,\mathrm{abs}}(M,g)\hookrightarrow L^{p}\Omega^k(M,g).$$
Moreover the above inclusion induces  an injective linear map $$\beta_p:\overline{H}^k_{2,\max}(M,g)\rightarrow \overline{H}^k_{p,\max}(M,g)$$ for each $2\leq p\leq \infty$
\end{teo}
Some remarks to the above theorem are appropriate: 
in the third item above $\Delta_{k,\mathrm{abs}}$ and $\Delta_k^{\mathcal{F}}$ are two different self-adjoint extensions of the $k$-th Hodge Laplacian. The former is the operator induced by the $L^2$ maximal extension of the de Rham complex whereas the latter  is the Friedrichs extension. Moreover $\mathcal{H}^k_{\mathrm{abs}}(M,g)$ denotes the nullspace of $\Delta_{k,\mathrm{abs}}$ that, as we will recall later, is isomorphic to $\overline{H}^k_{2,\max}(M,g)$. Thus Th. \ref{tozzix} shows that under certain assumptions the reduced maximal $L^2$ cohomology of $(M,g)$ injects into the reduced maximal $L^p$ cohomology of $(M,g)$ for any $2\leq p\leq\infty$. In the second main result we provide conditions that assures the injectivty of $\gamma$:
\begin{teo}
\label{LpL2x}
Let $(M,g)$ be an open incomplete Riemannian manifold of dimension $m>2$. Assume that
\begin{itemize}
\item $\vol_g(M)<\infty$;
\item There exists $k\in \{0,...,m\}$ such that $\ell_{k-1}^-,\ell_k^-\in \mathcal{K}(M)$;
\item $\mathrm{Im}(d_{k-1,\max,2})$ is closed in $L^2\Omega^k(M,g)$; 
\item $\mathrm{Im}(d_{k-1,\max,q})$ is closed in $L^q\Omega^k(M,g)$ for some $2\leq q< \infty$;
\item We have a continuous inclusion $W^{1,2}_0(M,g)\hookrightarrow L^{\frac{2m}{m-2}}(M,g)$;
\item  The operator $$\Delta_{k,\mathrm{abs}}:L^2\Omega^k(M,g)\rightarrow L^2\Omega^k(M,g)\quad\quad \mathrm{equals}\quad\quad \Delta_k^{\mathcal{F}}:L^2\Omega^k(M,g)\rightarrow L^2\Omega^k(M,g)$$
and analogously the operator $$\Delta_{k-1,\mathrm{abs}}:L^2\Omega^{k-1}(M,g)\rightarrow L^2\Omega^{k-1}(M,g)\quad\quad \mathrm{equals}\quad\quad \Delta_{k-1}^{\mathcal{F}}:L^2\Omega^{k-1}(M,g)\rightarrow L^2\Omega^{k-1}(M,g).$$
\end{itemize}
Then the map $$\gamma: H^k_{q,\max}(M,g)\rightarrow H^k_2(M,g)$$ induced by the continuous inclusion $L^q\Omega^k(M,g)\hookrightarrow L^2\Omega^k(M,g)$  is injective. If moreover $\im(d_{k-1,\max,2})$ is closed in $L^2\Omega^k(M,g)$ then $$\gamma: H^k_{q,\max}(M,g)\rightarrow H^k_2(M,g)$$ is an isomorphism.
\end{teo}
At this point the reader may wonder why we assumed our manifolds to be incomplete. The reason lies on the fact that a complete manifold with finite volume that carries a Sobolev embedding is necessarily compact, see Rmk. \ref{car}. Since in the compact setting the above results are well known (at least for $1<p\leq q<\infty$) we focused directly on open and incomplete manifolds of finite volume. Moreover although at a first glance the above sets of assumptions could appear quite restrictive we will see in the third section that actually this is not the case, at least for the singular spaces we are interested in. 
Concerning the proofs of Th. \ref{tozzix} and Th. \ref{LpL2x}, a key role in both of them is played by the $L^p$-$L^q$ mapping properties of the heat operator $e^{-t\Delta_k^{\mathcal{F}}}$. In the last part of the second section we use again these mapping properties to investigate another question that arises whenever one deals with incomplete Riemannian manifolds: the $L^p$-Stokes theorem. Briefly we say that the $L^p$-Stokes theorem holds true at the level of $k$-forms on $(M,g)$ if $d_{k,p,\max}$ equals $d_{k,p,\min}$, see Def. \ref{StokesLp}. Regarding this question we proved what follows:
\begin{teo} 
\label{LpStokesx}
Let $(M,g)$ be an open and  incomplete Riemannian manifold of dimension $m>2$ such that 
\begin{itemize}
\item $\vol_g(M)<\infty$;
\item There is a continuous inclusion $W^{1,2}_0(M,g)\hookrightarrow L^{\frac{2m}{m-2}}(M,g)$;
\item $(M,g)$ is $q$-parabolic for some $2<q<\infty$;
\item There exists $k\in \{0,...,m\}$ such that $\ell_k^-,\ell_{k+1}^-\in \mathcal{K}(M)$;
\item $\Delta_{k,\mathrm{abs}}=\Delta_{k}^{\mathcal{F}}$ and $\Delta_{k+1,\mathrm{abs}}=\Delta_{k+1}^{\mathcal{F}}$ as unbounded and self-adjoint operators acting on $L^2\Omega^k(M,g)$ and $L^2\Omega^{k+1}(M,g)$, respectively. 
\end{itemize}
Then the $L^r$-Stokes theorem holds true on $L^r\Omega^{k}(M,g)$ for each $r\in [2,q]$. 
\end{teo}

Finally the third and last section collects various applications of the above theorems to compact Thom-Mather stratified pseudomanifolds   and complex projective varieties with isolated singularities. Here we endow the former spaces with an iterated conic metric while the latters come equipped with the Fubini-Study metric. As we will see there are many examples of such singular spaces where all the assumptions of the above theorems except that concerning the negative part of the tensor  $L_k$ are fulfilled.  Since in this framework the $L^p$-cohomology is isomorphic to a certain kind of intersection cohomology the next results show how requiring $\ell_k^-$ in the Kato class of $\reg(X)$ (and so in particular a possibly negative lower bound on $L_k$) provides already strong topological consequences. In order to state these applications we recall briefly that $\underline{m}$ stands for the lower middle perversity defined as $\underline{m}(k):=[(k-2)/2]$, $t$ is the top perversity defined as $t(k):=k-2$ and finally, for any fixed $r\geq 2$, $q_r(k)$ is perversity given by $q_r(k):=t(k)-p_r(k)$ with $p_r(k):=[[k/r]]$ and  $[[\bullet]]$ denoting the biggest integer number strictly smaller than $\bullet$. We have now all the ingredient to state the next

\begin{teo}
\label{09/07x}
Let $X$ be a compact and oriented smoothly  Thom-Mather-Witt stratified pseudomanifolds of dimension $m>2$. Let $g$ be  an iterated conic metric on $\reg(X)$ such that $d+d^t:L^2\Omega^{\bullet}(\reg(X),g)\rightarrow L^2\Omega^{\bullet}(\reg(X),g)$ is essentially self-adjoint.  We have the following properties:
\begin{enumerate}
\item If $\ell_k^-\in \mathcal{K}(\reg(X))$ for some $k\in \{0,...,m\}$ then Th. \ref{tozzix} holds true for $(\reg(X),g)$ and $k$. As a consequence $$\dim(I^{\underline{m}}H^k(X,\mathbb{R}))\leq \dim(I^{q_r}H^k(X,\mathbb{R}))$$ for any $2\leq r<\infty$. If $X$ is also normal then $$\dim(I^{\underline{m}}H^k(X,\mathbb{R}))\leq \dim(H^k(X,\mathbb{R})).$$
\item Let $2\leq r< \infty$ be arbitrarily fixed. If  $\ell_{k-1}^-,\ell_k^-\in \mathcal{K}(\reg(X))$ for some $k\in \{0,...,m\}$ then the map $$\gamma:H^k_{r,\max}(\reg(X),g)\rightarrow H^k_{2}(\reg(X),g)$$ is injective, see Th. \ref{LpL2x}. Consequently $$\dim(I^{\underline{m}}H^k(X,\mathbb{R}))= \dim(I^{q_r}H^k(X,\mathbb{R}))$$ for any $2\leq r<\infty$. If $X$ is also normal then $$\dim(I^{\underline{m}}H^k(X,\mathbb{R}))= \dim(H^k(X,\mathbb{R})).$$
\item Assume that there exists $r\in (2,\infty)$ such that each singular stratum $Y\subset X$ satisfies $\mathrm{cod}(Y)\geq r$ if $\mathrm{depth}(Y)=1$ whereas $\mathrm{cod}(Y)> r$ if $\mathrm{depth}(Y)>1$. If $\ell_k^-,\ell_{k+1}^-\in \mathcal{K}(\reg(X))$ for some $k\in \{0,...,m\}$ then the $L^z$-Stokes theorem holds true on $L^z\Omega^k(M,g)$ for any $z\in [2,r]$. 
\end{enumerate}
\end{teo}

Concerning complex projective varieties with only isolated singularities whose regular part is endowed with the Fubini-Study metric we have the following

\begin{teo}
\label{projvarx}
Let $V$ be a complex projective variety of complex dimension $v>1$ with $\dim(\sing(V))=0$ and let  $h$ be the K\"ahler metric on $\reg(V)$ induced by the Fubini-Study metric. Assume that $\ell_k^-\in \mathcal{K}(\reg(V))$ for some $k\in \{0,...,2v\}$, $k\notin \{v\pm1, v\}$. Then
  $$\dim(I^{\underline{m}}H^k(V,\mathbb{R}))\leq \dim(I^{t}H^k(V,\mathbb{R})).$$  If $V$ is normal then $$\dim(I^{\underline{m}}H^k(V,\mathbb{R}))\leq \dim(H^k(V,\mathbb{R})).$$ If  $1<k<v-1$ then $$ \dim(H_{2v-k}(V,\mathbb{R}))\leq \dim(H_{2v-k}(\reg(V),\mathbb{R}))$$ whereas if $k=1$ we have $$\dim(H_{2v-1}(V),\mathbb{R})=\dim(\im(H_{2v-1}(\reg(V),\mathbb{R})\rightarrow H_{2v-1}(V,\mathbb{R}))).$$
Finally if we assume  that $L_k\geq 0$ and  $L_{k,p}> 0$ for some $p\in \reg(V)$, then $$I^{\underline{m}}H^k(V,\mathbb{R})=\{0\}.$$
\end{teo}

Besides the aforementioned  results the third section contains other applications. In particular we analyze the case $k=1$ that is, when the curvature term in the above theorems is the Ricci curvature, and we study in detail the case of isolated conical singularities. For the sake of brevity we refer directly to the third section for precise statements.

\vspace{1 cm}

\noindent\textbf{Acknowledgments.} It is a pleasure to thank Simone Diverio, Batu G\"uneysu, Kieran O'Grady and Paolo Piazza for helpful comments.
 
\section{Background material}
\subsection{$L^p$-cohomology and differential operators}
The aim of this section is to  recall briefly some basic notions about $L^p$-spaces, differential operators and $L^p$-cohomology.  We refer for instance  to  \cite{GoTro}, \cite{GoTr}, \cite{Batu}, \cite{PiRiSe} and \cite{BYOU}. Let $(M,g)$  be an open and possibly incomplete Riemannian manifold of dimension $m$. Throughout the text $M$ will always assumed to be connected. Let $E$ be a vector bundle over $M$ of rank $k$ and let $\rho$ be a metric on $E$, Hermitian if $E$ is a complex vector bundle, Riemannian if $E$ is a real vector bundle. Let $\dvol_g$ be the one-density associated to $g$.  We consider $M$ endowed with the canonical  Riemannian measure, see e.g.  \cite[Th. 3.11]{GYA}.  A section $s$ of $E$ is said {\em measurable} if, for any trivialization $(U,\phi)$ of $E$, $\phi(s|_U)$ is given by a $k$-tuple of measurable functions.  Given a measurable section $s$ let $|s|_{\rho}$ be defined as $|s|_{\rho}:= (\rho(s,s))^{1/2}$. Then for every $p$ with $1\leq p< \infty$ we define $L^{p}(M,E,g,\rho)$ as the  space  of equivalence classes of measurable sections $s$ such that    $$\|s\|_{L^{p}(M,E,g,\rho)}:=\left(\int_{M}|s|_{\rho}^p\dvol_g\right)^{1/p}<\infty.$$
For each $p\in [1, \infty)$ we have a Banach space,  for each $p\in (1, \infty)$  we get a reflexive Banach space and  in the case $p=2$ we have a Hilbert space whose inner product is $$\langle s, t \rangle_{L^2(M,E,g,\rho)}:= \int_{M}\rho(s,t)\dvol_g.$$ Moreover $C^{\infty}_c(M,E)$,  the space of smooth sections with compact support,  is dense in $L^p(M,E,g,\rho)$ for $p\in [1,\infty).$  Finally $L^{\infty}(M,E,\rho)$ is defined as the space of equivalence classes of measurable sections whose {\em essential\ supp} is bounded, that is the space of measurable sections $s$ such that $|s|_{\rho}$ is bounded almost everywhere. Also in this case we get a Banach space.  
Given any $p\in [1,\infty]$, $s \in L^p(M,E,g,\rho)$  and $v\in L^{p'}(M,E,g,\rho)$, with  $p'=\frac{p}{p-1}$ if $p>1$ whereas $p'=\infty$ if $p=1$, we have $$\rho(s,v)\in L^1(M,g)\quad \mathrm{and}\quad  \|\rho(s,v)\|_{L^1(M,g)}\leq \|s\|_{L^p(M,E,g,\rho)}\|v\|_{L^{p'}(M,E,g,\rho)}.$$ Moreover if $p\in (1,\infty)$ the map $T:L^{p'}(M,E,g,\rho)\rightarrow (L^p(M,E,g,\rho))^*$ given by $$v\mapsto Tv,\quad\quad Tv(s)=\int_M\rho(v,s)\dvol_g$$ is an isometric isomorphism between $L^{p'}(M,E,g,\rho)$ and  $(L^p(M,E,g,\rho))^*$. Consider now   another vector bundle $F$ over $M$ endowed with a metric $\tau$. Let $P: C^{\infty}_c(M,E)\longrightarrow  C^{\infty}_c(M,F)$ be a differential operator of order $d\in \mathbb{N}$. Then the {\em formal adjoint} of $P$ $$P^t: C^{\infty}_c(M,F)\longrightarrow C^{\infty}_c(M,E)$$ is the differential operator uniquely characterized by the following property: for each $u\in C^{\infty}_c(M,E)$ and for each $v\in C^{\infty}_c(M,F)$ we have  $$\int_{M}\rho(u, P^tv)\dvol_g=\int_M\tau(Pu,v)\dvol_g.$$ We can look at $P$ as an unbounded, densely defined and closable operator  acting between $L^p(M,E,g,\rho)$ and $L^p(M,F,g,\tau)$ with $1\leq p<\infty$. In general $P$ admits several different closed extensions. For our scopes we  recall now the  definition of the maximal and the minimal one. The domain of the {\em maximal extension} of $P:L^p(M,E,g,\rho)\longrightarrow L^p(M,F,g,\tau)$ is defined as
\begin{align}
\label{ner}
& \mathcal{D}(P_{\max,q}):=\{s\in L^p(M,E,g,\rho): \text{there is}\ v\in L^p(M,F,g,\tau)\ \text{such that}\ \int_{M}\rho(s,P^t\phi)\dvol_g=\\
& \nonumber =\int_{M}\tau(v,\phi)\dvol_g\ \text{for each}\ \phi\in C^{\infty}_c(M,F)\}.\ \text{In this case we put}\ P_{\max,q}s=v.
\end{align} In other words the maximal extension of $P$ is the one defined in the {\em distributional sense}. The domain of the {\em minimal extension} of $P:L^p(M,E,g,\rho)\longrightarrow L^p(M,F,g,\tau)$ is defined as
\begin{align}
\label{spinaci}
& \mathcal{D}(P_{\min,q}):=\{s\in L^p(M,E,g,\rho)\ \text{such that there is a sequence}\ \{s_i\}\in C_c^{\infty}(M,E)\ \text{with}\ s_i\rightarrow s\\ \nn & \text{in}\ L^p(M,E,g,\rho)\ \text{and}\ Ps_i\rightarrow w\ \text{in}\ L^p(M,F,g,\tau)\ \text{to some }\ w\in L^p(M,F,g,\tau)\}.\ \text{We put}\ P_{\min,q}s=w.
\end{align} Briefly the minimal extension of $P$ is the closure of $C^{\infty}_c(M,E)$ under the graph norm $\|s\|_{L^p(M,E,g,\rho)}+\|Ps\|_{L^p(M,F,g,\tau)}$. Clearly $\mathcal{D}(P_{\min,p})\subset \mathcal{D}(P_{\max,p})\ \mathrm{and}\ P_{\max,p}s=P_{\min,p}s$ for any $s\in \mathcal{D}(P_{\min,p})$. Moreover it is easy to check that $Ps= P_{\max,q}s$ for any $s\in C^{\infty}(M,E)\cap \mathcal{D}(P_{\max,q})$ and analogously $Ps=P_{\min,q}s$ for any $s\in C^{\infty}(M,E)\cap \mathcal{D}(P_{\min,q})$. In other words the action of $P_{\max/\min,q}$ on smooth sections lying in $\mathcal{D}(P_{\max/\min,q})$ coincides with the standard one.  Another well known and important property is the following: given $1<p<\infty$ and $p'=\frac{p}{p-1}$ we have 
\begin{equation}
\label{Badjoint}
(P_{p,\max})^*=P^t_{p',\min}\ \mathrm{and}\ (P_{p,\min})^*=P^t_{p',\max}
\end{equation}
that is $P^t_{p',\min}:L^{p'}(M,F,g,\tau)\rightarrow L^{p'}(M,E,g,\rho)$ is the Banach adjoint of $P_{p,\max}:L^{p}(M,E,g,\rho)\rightarrow L^{p}(M,F,g,\tau)$ and similarly $P^t_{p',\max}:L^{p'}(M,F,g,\tau)\rightarrow L^{p'}(M,E,g,\rho)$ is the Banach adjoint of $P_{p,\min}:L^{p}(M,E,g,\rho)\rightarrow L^{p}(M,F,g,\tau)$. See e.g. \cite[Lemma I.22]{Batu} for a proof of \eqref{Badjoint} and \cite{TKato} for the general notion of Banach adjoint.\\
Also in the case $p=\infty$ $P$ admits a closed extension $P_{\max,\infty}:L^{\infty}(M,E,\rho)\longrightarrow L^{\infty}(M,F,\tau)$ defined as
\begin{align}
\label{nerop}
& \mathcal{D}(P_{\max,\infty}):=\{s\in L^{\infty}(M,E,\rho): \text{there is}\ v\in L^{\infty}(M,F,\tau)\ \text{such that}\ \int_{M}\rho(s,P^t\phi)\dvol_g=\\
& \nonumber =\int_{M}\tau(v,\phi)\dvol_g\ \text{for each}\ \phi\in C^{\infty}_c(M,F)\}.\ \text{In this case we put}\ P_{\max,\infty}s=v.
\end{align} 
Put it differently a measurable section $s\in L^{\infty}(M,E,\rho)$ lies in $\mathcal{D}(P_{\max,\infty})$ if the distributional action of $P$ applied to $s$ lies in $L^{\infty}(M,F,\tau)$. We collect now some well known facts that will be frequently used later.\\

\noindent \textbf{a)} Let $1\leq p\leq q \leq \infty$. Assume that $\vol_g(M)<\infty$. It is well known that if $s\in L^q(M,E,g,\rho)$ then $s\in L^p(M,E,g,\rho)$ and that the corresponding inclusion $i:L^q(M,E,g,\rho)\hookrightarrow L^p(M,E,g,\rho)$ is continuous. Then it is immediate to check that if $s\in \mathcal{D}(P_{\max,q})$ then $s=i(s)\in \mathcal{D}(P_{\max,p})$ and 
\begin{equation}
\label{corno}
i\circ P_{\max,q}=P_{\max,p}\circ i\quad\quad \text{on}\ \mathcal{D}(P_{\max,q}).
\end{equation}
Similarly if  $1\leq p\leq q < \infty$ and $s\in \mathcal{D}(P_{\min,q})$ then $s=i(s)\in \mathcal{D}(P_{\min,p})$ and 
\begin{equation}
\label{delcatria}
i\circ P_{\min,q}=P_{\min,p}\circ i\quad\quad \text{on}\ \mathcal{D}(P_{\min,q}).
\end{equation}

\noindent \textbf{b)} Let $1\leq p\leq q < \infty$ and assume that $\vol_g(M)<\infty$. Let $\{s_n\}_{n\in\mathbb{N}}$ be any sequence such that $s_n\rightharpoonup s$ in $L^q(M,E,g,\rho)$ as $n\rightarrow \infty$, that is $\{s_n\}$ converges weakly to some $s\in L^q(M,E,g,\rho)$ as $n\rightarrow \infty$. Then  
\begin{equation}
\label{weak}
s_n\rightharpoonup s\ \mathrm{in}\ L^p(M,E,g,\rho)
\end{equation}
as $n\rightarrow \infty$. We refer to \cite{TKato} for the definition of weak convergence.\\

\noindent \textbf{c)} Let $1\leq p\leq q \leq \infty$ and let $U\subset M$ be an open subset. Let $s\in \mathcal{D}(P_{\max,p})\subset L^p(M,E,g,\rho)$. Then 
\begin{equation}
\label{otoedoda}
s|_U\in \mathcal{D}({P|_U}_{\max,p})\subset L^p(U,E|_U)\ \text{and}\ 
{P|_U}_{\max,p}(s|_U)=P_{\max,p}s|_U.
\end{equation}

\noindent \textbf{d)} Let $1<p<\infty$, let $p'=p/(p-1)$ and let $s\in L^p(M,E,g,\rho)$. Then $s\in \mathcal{D}(P_{\max,p})$ and $L^{p}(M,F,g,\tau)\ni w=P_{\max,p}s$ if and only if for any $v\in \mathcal{D}(P^t_{\min,p'})$ we have 
\begin{equation}
\label{cope}
\int_M\tau(w,v)\dvol_g=\int_{M}\rho(s, P^t_{\min,p'}v)\dvol_g.
\end{equation}
Analogously, given an arbitrary  $s\in L^p(M,E,g,\rho)$, we have $s\in \mathcal{D}(P_{\min,p})$ and $L^p(M,F,g,\tau)\ni w=P_{\min,p}s$ if and only if for any $v\in \mathcal{D}(P^t_{\max,p'})$ we have 
\begin{equation}
\label{naghen}
\int_M\tau(w,v)\dvol_g=\int_{M}\rho(s, P^t_{\max,p'}v)\dvol_g.
\end{equation}

\noindent \textbf{e)} Let $1<p<\infty$, let $p'=p/(p-1)$ and let $s\in L^p(M,E,g,\rho)$. Then $s\in \ker(P_{\max/\min,p})$  if and only if for any $v\in \mathcal{D}(P^t_{\min/\max,p'})$ we have 
\begin{equation}
\label{co}
\int_{M}\rho(s, P^t_{\min/\max,p'}v)\dvol_g=0.
\end{equation}
Given any  $w\in L^p(M,F,g,\tau)$, we have $w\in \overline{\mathrm{im}(P_{\min/\max,p})}$ and  if and only if for any $v\in \ker(P^t_{\max/\min,p'})$ we have 
\begin{equation}
\label{vid}
\int_M\tau(w,v)\dvol_g=0.
\end{equation}
Finally $\mathrm{im}(P_{\min/\max})$ is closed in $L^p(M,F,g,\tau)$ if and only if $\mathrm{im}(P^t_{\max/\min})$ is closed in $L^{p'}(M,E,g,\tau)$.\\

\noindent \textbf{f)} Consider the case $p=2$. The operator $S:=P^t\circ P:L^2(M,E,g,\rho)\rightarrow L^2(M,E,g,\rho)$ is non-negative and formally self-adjoint and it owns at least a self-adjoint extension: the so called {\em Friedrichs extension}. We denote it by   $S^{\mathcal{F}}:L^2(M,E,g,\rho)\rightarrow L^2(M,E,g,\rho)$  and the corresponding domain is given by
 \begin{align}
& \nonumber \mathcal{D}(S^{\mathcal{F}})=\{\omega\in \mathcal{D}(P_{\min})\cap \mathcal{D}(S_{\max})\}
\end{align}
Equivalently $\mathcal{D}(S^{\mathcal{F}})$ is given by 
\begin{equation}
\label{Fred}
 \mathcal{D}(S^{\mathcal{F}})=\{\omega\in \mathcal{D}(P_{\min})\ \mathrm{such\ that}\ P_{\min}\omega\in \mathcal{D}(P^t_{\max})\}.
\end{equation}

\noindent From now on we focus on the de Rham differential. Consider $\Lambda^kT^*M$, the $k$-th exterior power of the cotangent bundle, and with a little abuse of notations let us denote by $g$ the metric induced by $g$ on $\Lambda^kT^*M$. We  denote the corresponding $L^p$ space as  $L^p\Omega^k(M,g)$. In the case of the trivial bundle $M\times \mathbb{R}$ we  simply write $L^p(M,g)$. Accordingly to the previous definitions with  $d_{k,\max/\min,q}:L^q\Omega^k(M,g)\rightarrow L^q\Omega^{k+1}(M,g)$, $1\leq q<\infty$, we denote the maximal/minimal extension of $d_k$ acting on $L^q\Omega^k(M,g)$ while with $d_{k,\max,\infty}:L^{\infty}\Omega^k(M,g)\rightarrow L^{\infty}\Omega^{k+1}(M,g)$, we denote the maximal extension of $d_k$ acting on $L^{\infty}\Omega^k(M,g)$. 
It is easy to verify that if $\omega\in \mathcal{D}(d_{k,\max/\min,p})$ then $d_{k,\max/\min,p}\omega\in \mathcal{D}(d_{k+1,\max/\min,p})$ and the corresponding compositions are identically zero, that is $d_{k+1,\max,p}\circ d_{k,\max,p}\equiv 0$ and $d_{k+1,\min,p}\circ d_{k,\min,p}\equiv 0$. Analogously $d_{k+1,\max,\infty}\circ d_{k,\max,\infty}\equiv 0$ on $\mathcal{D}(d_{k,\max,\infty})$. The $L^p$-maximal/minimal de Rham cohomology of $(M,g)$ is thus defined as $$H^k_{p,\max/\min}(M,g):=\ker(d_{k,\max/\min,p})/\im(d_{k-1,\max/\min,p})$$ while the reduced $L^p$-maximal/minimal de Rham cohomology of $(M,g)$ is  defined as 
\begin{equation}
\label{gintonic}
\overline{H}^k_{p,\max/\min}(M,g):=\ker(d_{k,\max/\min,p})/\overline{\im(d_{k-1,\max/\min,p})},
\end{equation}
where $\overline{\im(d_{k-1,\max/\min,p})}$ is the closure of $\im(d_{k-1,\max/\min,p})$ in $L^p\Omega^k(M,g)$ respectively.
Clearly the identity $\ker(d_{k,\max/\min,p})\rightarrow \ker(d_{k,\max/\min,p})$ induces a surjective map 
$$H^k_{p,\max/\min}(M,g)\rightarrow \overline{H}^k_{p,\max/\min}(M,g).$$
If $1<p<\infty$ and $H^k_{p,\max/\min}(M,g)$ is finite dimensional then $\mathrm{im}(d_{k-1,\max/\min,p})$ is closed in $L^p\Omega^k(M,g)$ and $H^k_{p,\max/\min}(M,g)=\overline{H}^k_{p,\max/\min}(M,g).$
When $M$ is oriented,  $1<p<\infty$ and $p'=\frac{p}{p-1}$ we have the following important properties: the bilinear map given by 
\begin{equation}
\label{LpPoinca}
\overline{H}^k_{p,\max}(M,g)\times \overline{H}^{m-k}_{p',\min}(M,g)\rightarrow \mathbb{R},\ ([\omega],[\eta])\mapsto \int_M\omega\wedge\eta
\end{equation}
is a well-defined and non-degenerate pairing, see i.e \cite{Gold} and \cite{GoTro}. Moreover it is easy to check that given a form $\omega\in L^p\Omega^k(M,g)$ we have $\omega\in \mathcal{D}(d_{k,\max/\min,p})$ if and only if $*\omega\in \mathcal{D}(d^t_{m-k-1,\max/\min,p'})$ and $d_{k,\max/\min,p}\omega=\pm d^t_{m-k-1,\max/\min,p'}*\omega$ with the sign depending on the parity of $k$. We give now the following 
\begin{defi}
\label{StokesLp}
Let $(M,g)$ be a possibly incomplete Riemannian manifold. Let $1\le p<\infty$. We will say that the $L^p$-Stokes theorem holds on $L^p\Omega^k(M,g)$ if the following two operators $$d_{k,\max,p}:L^p\Omega^k(M,g)\rightarrow L^p\Omega^{k+1}(M,g)\quad\quad \text{and}\quad\quad d_{k,\min,p}:L^p\Omega^k(M,g)\rightarrow L^p\Omega^{k+1}(M,g)$$ coincide.
\end{defi}
It is well known  that if $(M,g)$ is complete then the $L^p$-Stokes theorem holds true for any $k=0,...,m$, see \cite{PiRiSe}. Conversely there are examples of incomplete Riemannian manifolds where the $L^p$-Stokes theorem fails to be true, see i.e \cite{GL}. In case the $L^p$-Stokes theorem holds true we will delete the subscript $\min/\max$ from the operators and the cohomology groups and  we will simply write $d_{k,p}$, $\overline{H}^k_p(M,g)$ and $H^k_p(M,g)$. Consider now the case $p=\infty$. Similarly to the previous case  we have the $L^{\infty}$-de Rham cohomology and the reduced $L^{\infty}$-de Rham cohomology defined respectively as $$H^k_{\infty}(M,g):=\ker(d_{k,\max,\infty})/\im(d_{k-1,\max,\infty})$$ and 
\begin{equation}
\label{moscowmule}
\overline{H}^k_{\infty}(M,g):=\ker(d_{k,\max,\infty})/\overline{\im(d_{k-1,\max,\infty})}.
\end{equation} 
Clearly also in this case we have a surjective map 
$$H^k_{\infty,\max}(M,g)\rightarrow \overline{H}^k_{\infty,\max}(M,g)$$
 induced by the identity $\ker(d_{k,\max,\infty})\rightarrow \ker(d_{k,\max,\infty})$. For our purposes we need also to introduce the following variants of the $L^{\infty}$-cohomology. Let $(M,g)$ be an arbitrary Riemannian manifold.  
We define $\Omega_{s,\infty}^k(M,g):=\mathcal{D}(d_{k,\max,\infty})\cap \Omega^k(M)$ and we consider the following complex
\begin{equation}
\label{smooth}
...\stackrel{d_{k-1}}{\rightarrow}\Omega^k_{s,\infty}(M,g)\stackrel{d_k}{\rightarrow}\Omega^{k+1}_{s,\infty}(M,g)\stackrel{d_{k+1}}{\rightarrow}...
\end{equation}
We denote by $H_{s,\infty}^k(M,g)$ the corresponding cohomology groups. Here the subscript $s$ emphasizes the fact that the complex is built using smooth forms. In particular for any $\omega\in \Omega^k_{s,\infty}(M,g)$ we have $d_k\omega=d_{k,\max,\infty}\omega$. Finally we conclude with the following 

\begin{defi}
\label{parax}
Let $(M,g)$ be a  possibly incomplete Riemannian manifold.  Let $q\in [1,\infty)$. Then $(M,g)$ is said to be $q$-parabolic if there exists a sequence of Lipschitz functions with compact support $\{\phi_n\}_{n\in \mathbb{N}}\subset \mathrm{Lip}_c(M,g)$ such that
\begin{enumerate}
\item $0\leq \phi_n\leq 1$;
\item $\phi_n\rightarrow 1$ almost everywhere as $n\rightarrow \infty$;
\item $\|d_0\phi_n\|_{L^q\Omega^1(M,g)}\rightarrow 0$ as $n\rightarrow \infty$.
\end{enumerate}
\end{defi}
We invite the reader to consult \cite{GYA} and \cite{paraTroya} for an in-depth treatment about $q$-parabolicity.

\subsection{Hodge Laplacian and heat operator}
Let us consider again an open and possibly incomplete Riemannian manifold $(M,g)$ of dimension $m$ and let $\Delta_k:\Omega_c^k(M)\rightarrow \Omega_c^k(M)$, $\Delta_k:=d^t_k\circ d_k+d_{k-1\circ}d^t_{k-1}$, be the Hodge Laplacian acting on smooth $k$-forms with compact support. It is well known that $\Delta_k$, viewed as an unbounded and densely defined operator acting on $L^2\Omega^k(M,g)$, is formally self-adjoint. Clearly we can rewrite $\Delta_k$ as $\Delta_k= D^t_k\circ D_k$ with $$D_k:=d_{k-1}^t+d_{k}:\Omega_c^k(M)\rightarrow \Omega_c^{k-1}(M)\oplus \Omega_c^{k+1}(M)\quad \mathrm{and}\quad  D_k^t:=d_{k-1}+d_{k}^t: \Omega_c^{k-1}(M)\oplus \Omega_c^{k+1}(M)\rightarrow\Omega_c^k(M).$$ Hence $\Delta_k:\Omega_c^k(M)\rightarrow L^2\Omega^k(M,g)$ admits at least a self-adjoint extension: the Friedrichs extension  $\Delta_k^{\mathcal{F}}:L^2\Omega^k(M,g)\rightarrow L^2\Omega^k(M,g)$.  Following \cite{BL} we recall now the definitions of other two important self-adjoint extensions of $\Delta_k$. The first  is the so-called {\em absolute} extension, denoted with $\Delta_{k,\mathrm{abs}}:L^2\Omega^k(M,g)\rightarrow L^2\Omega^k(M,g)$. Its domain is given by 
\begin{align}
& \nonumber \mathcal{D}(\Delta_{\mathrm{abs},k})=\{\omega\in \mathcal{D}(d_{k,\max,2})\cap\mathcal{D}(d^t_{k-1,\min,2}) \ \mathrm{such\ that}\ d_{k,\max,2}\omega\in \mathcal{D}(d^t_{k,\min,2})\ \mathrm{and}\ d^t_{k-1,\min,2}\omega\in \mathcal{D}(d_{k-1,\max,2})\}.
\end{align}
The absolute extension is the self-adjoint extension of $\Delta_k$ induced by the $L^2$-maximal de Rham complex. It satisfies the following properties $$\ker(\Delta_{k,\mathrm{abs}})=\ker(d_{k,\max,2})\cap \ker(d^t_{k-1,\min})\quad \mathrm{and}\quad \overline{\mathrm{im}(\Delta_{k,\mathrm{abs}})}=\overline{\mathrm{im}(d_{k-1,\max,2})}\oplus \overline{\mathrm{im}(d^t_{k,\min,2})}$$ and for the corresponding $L^2$-orthogonal decomposition of $L^2\Omega^k(M,g)$ we have $$L^2\Omega^k(M,g)=\ker(\Delta_{k,\mathrm{abs}})\oplus \overline{\mathrm{im}(\Delta_{k,\mathrm{abs}})}=\mathcal{H}^k_{2,\mathrm{abs}}(M,g)\oplus \overline{\mathrm{im}(d_{k-1,\max,2})}\oplus \overline{\mathrm{im}(d^t_{k,\min,2})}$$ where $\mathcal{H}^k_{2,\mathrm{abs}}(M,g):=\ker(\Delta_{k,\mathrm{abs}})=\ker(d_{k,\max,2})\cap \ker(d^t_{k-1,\min,2})$. In particular we have the following isomorphism 
\begin{equation}
\label{isoabs}
\mathcal{H}^k_{2,\mathrm{abs}}(M,g)\cong \overline{H}^k_{2,\max}(M,g).
\end{equation}
 Now we pass to the 
 {\em relative} extension, denoted with $\Delta_{k,\mathrm{rel}}:L^2\Omega^k(M,g)\rightarrow L^2\Omega^k(M,g)$. Its domain is given by 
\begin{align}
& \nonumber \mathcal{D}(\Delta_{\mathrm{rel},k})=\{\omega\in \mathcal{D}(d_{k,\min,2})\cap\mathcal{D}(d^t_{k-1,\max,2}) \ \mathrm{such\ that}\ d_{k,\min,2}\omega\in \mathcal{D}(d^t_{k,\max,2})\ \mathrm{and}\ d^t_{k-1,\max,2}\omega\in \mathcal{D}(d_{k-1,\min,2})\}.
\end{align}
It is the self-adjoint extension of $\Delta_k$ induced by the $L^2$-minimal de Rham complex. In analogy to the previous case we have the following properties $$\ker(\Delta_{k,\mathrm{rel}})=\ker(d_{k,\min,2})\cap \ker(d^t_{k-1,\max})\quad \mathrm{and}\quad \overline{\mathrm{im}(\Delta_{k,\mathrm{rel}})}=\overline{\mathrm{im}(d_{k-1,\min,2})}\oplus \overline{\mathrm{im}(d^t_{k,\max,2})}$$ and for the corresponding $L^2$-orthogonal decomposition of $L^2\Omega^k(M,g)$ we have $$L^2\Omega^k(M,g)=\ker(\Delta_{k,\mathrm{rel}})\oplus \overline{\mathrm{im}(\Delta_{k,\mathrm{rel}})}=\mathcal{H}^k_{2,\mathrm{rel}}(M,g)\oplus \overline{\mathrm{im}(d_{k-1,\min,2})}\oplus \overline{\mathrm{im}(d^t_{k,\max,2})}$$ where $\mathcal{H}^k_{2,\mathrm{rel}}(M,g):=\ker(\Delta_{k,\mathrm{rel}})=\ker(d_{k,\min,2})\cap \ker(d^t_{k-1,\max,2})$. In particular we have the following isomorphism $$\mathcal{H}^k_{2,\mathrm{rel}}(M,g)\cong \overline{H}^k_{2,\min}(M,g).$$ In case $\Delta_{k,\mathrm{abs}}=\Delta_{k,\mathrm{rel}}$ we will simply denote its kernel with $\mathcal{H}^k_2(M,g)$. We refer to \cite{BL} for the above definitions and properties. In the next proposition we collect other properties concerning these self-adjoint extensions.
\begin{prop}
\label{oneto}
Let $(M,g)$ be an open and possible incomplete Riemannian manifold of dimension $m$. Then 
\begin{enumerate}
\item $\mathcal{D}(\Delta_{k,\mathrm{abs/rel}})$ is dense in $\mathcal{D}(d_{k,\mathrm{\max/\min},2})$ with respect to the corresponding graph norm;
\item $d_{k,\max,2}=d_{k,\min,2}$ if and only if $\Delta_{k,\mathrm{abs}}=\Delta_{k,\mathrm{rel}}$;
\item $\mathcal{D}(\Delta_{k}^{\mathcal{F}})$ is dense in $\mathcal{D}((d_k+d^t_{k-1})_{\min})$ with respect to the corresponding graph norm;
\item If $(d_k+d_{k-1}^t)_{\max}=(d_{k}+d_{k-1}^t)_{\min}$ then $\Delta_{k}^{\mathcal{F}}=\Delta_{k,\mathrm{abs}}=\Delta_{k,\mathrm{rel}}$.
\item If $\Delta_k^{\mathcal{F}}=\Delta_{k,\mathrm{abs}}$ then $d_{k,\max,2}$ $=d_{k,\min,2}$.
\end{enumerate}
\end{prop}
\begin{proof}
The first and the second  property  are  proved in \cite[Lemma 3.2]{BLE}. The third  property follows by \cite[Prop. 2.1]{FrB} and the fourth property is proved in \cite[Lemma 3.3]{BLE}. Finally the fifth property follows by \cite[Lemma 3.4]{BLE}.
\end{proof}

Finally we conclude this section with  the following property that will be used frequently later on. For the definition of the heat operator (or heat semigroup) we refer to \cite{Batu}. 

\begin{prop}
\label{commu}
Let $(M,g)$ be an open and possibly incomplete Riemannian manifold of dimension $m$. If $\omega\in \mathcal{D}(d_{k,\max/\min,2})$ then $e^{-t\Delta_{k,\max/\min,k}}\omega\in \mathcal{D}(d_{k,\max/\min,2})$ and $$d_{k,\max/\min,2}e^{-t\Delta_{k,\max/\min,k}}\omega=e^{-t\Delta_{k,\max/\min,k}}d_{k,\max/\min,2}\omega.$$
\end{prop}

\begin{proof}
See \cite[Prop. XIII.7]{Batu}.
\end{proof}

\subsection{The Kato class}
In this subsection we recall briefly the notion of Kato class and some of its properties.  For an in-depth treatment we refer to \cite[Chap. VI-VII]{Batu} and the reference therein. Let us consider again an open and possibly incomplete Riemannian manifold $(M,g)$ of dimension $m$ and let $f:M\rightarrow \mathbb{R}$ be a measurable function. We say that $f$ belongs to the {\em Kato class} of $(M,g)$, briefly $f\in \mathcal{K}(M)$, if
\begin{equation}
\label{katoclass}
\lim_{t\rightarrow 0^+}\sup_{x\in M}\int_0^t\int_Mp(s,x,y)|f(y)|\dvol_g(y)ds=0
\end{equation}
with $(s,x,y)\in (0,\infty)\times M\times M$ and $p(s,x,y)$ denoting the smooth kernel of the heat operator $e^{-t\Delta_0^{\mathcal{F}}}:L^2(M,g)\rightarrow L^2(M,g)$.  
According to \cite[Lemma VI.3]{Batu} we have 
\begin{equation}
\label{bKato}
L^{\infty}(M)\subset \mathcal{K}(M)\subset L^1_{\mathrm{loc}}(M).
\end{equation}
Moreover if there is a $C>0$ and a $T>0$ such
that for all  $0<t<T$ one has 
$$\sup_{x\in M}p(t,x,x)\leq Ct^{-\frac{m}{2}}$$ then it holds 
\begin{equation}
L^{q'}(M,g)+L^{\infty}(M)\subset \mathcal{K}(M)
\end{equation}
 for any $q'\geq 1$ if $m=1$, and $q'>m/2$ if $m\geq 2$, see \cite[Prop. VI.10 and Cor. VI.13]{Batu}. Let now $\nabla:C^{\infty}(M,TM)\rightarrow C^{\infty}(M,T^*M\otimes TM)$ be the Levi-Civita connection and with a little abuse of notation let us still label by $\nabla:C^{\infty}(M,\Lambda^k(M))\rightarrow C^{\infty}(M,T^*M\otimes \Lambda^k(M))$ the connection on $\Lambda^k(M)$ induced by the Levi-Civita connection. We endow $T^*M\otimes \Lambda^k(M)$ with the natural metric induced by $g$ and we denote by $\nabla^t:C^{\infty}(M,T^*M\otimes \Lambda^k(M))\rightarrow C^{\infty}(M,\Lambda^k(M))$ the formal adjoint of $\nabla:C^{\infty}(M,TM)\rightarrow C^{\infty}(M,T^*M\otimes TM)$. Finally let $\Delta_k:\Omega^k(M)\rightarrow \Omega^k(M)$ be the Hodge Laplacian acting on $k$-forms. 
The well known Weitzenb\"ock formula says that there exists $L_k\in C^{\infty}(M,\mathrm{End}(\Lambda^k(M)))$ such that \begin{equation}
\label{weit}
\Delta_k=\nabla^t\circ \nabla+L_k.
\end{equation} 
Let us define $\ell_k^-:M\rightarrow \mathbb{R}$ as 
\begin{equation}
\label{negativeKato}
\ell_k^-(x):=\max\{-l_k(x),0\}
\end{equation}
with  $$l_k(x):=\inf_{v\in \Lambda^kT_x^*M,\ g(v,v)=1}g(L_{k,x}v,v).$$
Note that $\ell^-_k(x)\geq 0$ for any $x\in M$. Clearly the following inequality holds true: for any $x\in M$ and $v\in \Lambda^kT_x^*M$ we have $$g(L_{k,x}v,v)\geq -\ell^-(x)g(v,v).$$
If $k=1$ it is well known that $L_1=\mathrm{Ric}$ where with a little abuse of notation $\mathrm{Ric}$ denotes  the endomorphism of $T^*M$ induced by $g$ and the Ricci tensor. In this case we will denote $\ell_1^-$ with $\mathrm{ric}^-$.
\begin{defi}
\label{Katodef}
In the setting above we say that the negative part of $L_k$ lies in the Kato class  if $$\ell^-_k(x)\in \mathcal{K}(M).$$
\end{defi}
Note that if $L_k\geq c$ for some $c\in \mathbb{R}$ then $0\leq \ell^-_k\leq |c|$ and thus $\ell_k^-\in \mathcal{K}(M)$ thanks to \eqref{bKato}. We have the following important result:

\begin{teo}
\label{Kato-Simon}
Let $(M,g)$ be an open and possibly incomplete Riemannian manifold of dimension $m$. Let $\Delta_k=\nabla^t\circ \nabla+L_k$ be the Hodge Laplacian acting on $k$-forms. If $\ell_k^-\in \mathcal{K}(M)$ then the following pointwise inequality $$|e^{-t\Delta_k^{\mathcal{F}}}\omega|_g\leq e^{-t\ell^-}e^{-t\Delta_0^{\mathcal{F}}}|\omega|_g$$ holds true for any $\omega\in L^2\Omega^k(M,g)$ and $x\in M$.
\end{teo}

\begin{proof}
This is a particular case of a more general result, indeed a Kato-Simon inequality for a large class of Schr\"odinger operators, proved in \cite[Th. VII.8]{Batu}.
\end{proof}

With the next result we recall some $L^q\rightarrow L^q$ and $L^p\rightarrow L^q$ properties of the heat operator $e^{-t\Delta_k^{\mathcal{F}}}$ when $\ell^-_k\in \mathcal{K}(M)$. 

\begin{teo}
\label{LpLqKato}
Let $(M,g)$ be an open and possibly incomplete Riemannian manifold of dimension $m$ such that $\ell_k^-\in \mathcal{K}(M)$.
\begin{enumerate}
\item For any $\delta>1$ there exists $c=c(\delta,\ell^-_k)\geq 0$ such that for all $t\geq 0$, $q\in [1,\infty]$ and $\omega\in L^2\Omega^k(M,g)\cap L^q\Omega^k(M,g)$ one has  $e^{-t\Delta_k^{\mathcal{F}}}\omega\in L^q\Omega^k(M,g)$ and 
\begin{equation}
\label{klqlq}
\|e^{-t\Delta_k^{\mathcal{F}}}\omega\|_{L^q\Omega^k(M,g)}\leq \delta e^{ct}\|\omega\|_{L^q\Omega^k(M,g)}.
\end{equation}
\item If in addition $m>2$ and we have a Sobolev embedding $$W^{1,2}_0(M,g)\hookrightarrow L^{\frac{2m}{m-2}}(M,g)$$ then there exist positive  constants $c_1$ and $c_2$ such that for any $1\leq p\leq q\leq \infty$,  $\omega\in L^2\Omega^k(M,g)\cap L^p\Omega^k(M,g)$ and $0<t\leq 1$ one has  $e^{-t\Delta_k^{\mathcal{F}}}\omega\in L^q\Omega^k(M,g)$ and 
\begin{equation}
\label{klplq}
\|e^{-t\Delta_k^{\mathcal{F}}}\omega\|_{L^q\Omega^k(M,g)}\leq c_1e^{c_2t}t^{-\frac{m}{2}(\frac{1}{p}-\frac{1}{q})}\|\omega\|_{L^p\Omega^k(M,g)}.
\end{equation}
\end{enumerate}
\end{teo}

\begin{proof}
For the first property see \cite[Cor. IX.4]{Batu}. Concerning the second property we first recall that a Sobolev embedding $W^{1,2}_0(M,g)\hookrightarrow L^{\frac{2m}{m-2}}(M,g)$ is equivalent to saying that there exists a positive constant $C$ such that $p(t,x,y)$, the kernel of $e^{-t\Delta_0^{\mathcal{F}}}:L^2(M,g)\rightarrow L^2(M,g)$, satisfies  $p(t,x,y)\leq Ct^{-m/2}$ for any $(x,y)\in M\times M$ and $0<t\leq1$, see \cite[Th. 6.3.1]{Gentil}. Now the conclusion follows by \cite[Th. IX.2 and Lemma VI.8]{Batu}. 
\end{proof}

\section{Some applications of the heat operator}
\subsection{Heat operator and $L^p$-cohomology}
We start with the following 
\begin{prop}
\label{tozzi}
Let $(M,g)$ be an open and possibly incomplete Riemannian manifold of dimension $m>2$. Assume that
\begin{itemize}
\item We have a continuous inclusion 
\begin{equation}
\label{imagine}
W^{1,2}_0(M,g)\hookrightarrow L^{\frac{2m}{m-2}}(M,g);
\end{equation}
\item There exists $k\in \{0,...,m\}$ such that $\ell^-_k\in \mathcal{K}(M)$.
\end{itemize}
 Then we have a continuous inclusion 
$$ \ker(\Delta_k^{\mathcal{F}})\hookrightarrow L^{q}\Omega^k(M,g) $$
 for each $q\in [2,\infty]$.
\end{prop}
\begin{proof}
Let $\omega\in \ker(\Delta_k^{\mathcal{F}})$. We have $\omega=e^{-t\Delta^{\mathcal{F}}_k}\omega$ and thus, thanks to \eqref{klplq}, we have $$\|\omega\|_{L^q\Omega^k(M,g)}=\|e^{-t\Delta_k^{\mathcal{F}}}\omega\|_{L^q\Omega^k(M,g)}\leq c_1e^{c_2t}t^{-\frac{m}{2}(\frac{1}{2}-\frac{1}{q})}\|\omega\|_{L^2\Omega^k(M,g)}$$ for any $0<t\leq 1$.
We can thus conclude that if $\omega\in \ker(\Delta_k^{\mathcal{F}})$ then $\omega\in L^q\Omega^k(M,g)$ and the resulting inclusion $$\ker(\Delta_k^{\mathcal{F}})\hookrightarrow L^{q}\Omega^k(M,g)$$ is continuous. 
\end{proof}

\begin{teo}
\label{tozziti}
In the setting of Prop. \ref{tozzi}. Assume in addition that $(M,g)$ is incomplete with  finite volume and the operator 
$$\Delta_{k,\mathrm{abs}}:L^2\Omega^k(M,g)\rightarrow L^2\Omega^{k}(M,g)\quad\quad equals\quad\quad  \Delta_{k}^{\mathcal{F}}:L^2\Omega^k(M,g)\rightarrow L^2\Omega^{k}(M,g).$$
Then for any $2\leq p\leq \infty$ we have a continuous inclusion 
\begin{equation}
\label{thanks}
\mathcal{H}^k_{2,\mathrm{abs}}(M,g)\hookrightarrow L^{p}\Omega^k(M,g).
\end{equation}
Moreover the above inclusion induces  an injective linear map $$\beta_p:\overline{H}^k_{2,\max}(M,g)\rightarrow \overline{H}^k_{p,\max}(M,g)$$ for each $2\leq p\leq \infty$ and an injective linear map $$\alpha_{\infty}:\overline{H}^k_{2,\max}(M,g)\rightarrow \im(H^k_{s,\infty}(M,g)\rightarrow\overline{H}^k_{\infty,\max}(M,g)).$$ 
\end{teo}

\begin{proof}
	Since we required $\Delta_k^{\mathcal{F}}=\Delta_{k,\mathrm{abs}}$ we have a continuous inclusion $\mathcal{H}^k_{2,\mathrm{abs}}(M,g)\hookrightarrow L^{p}\Omega^k(M,g)$ with $2\leq p\leq \infty$ thanks to Prop. \ref{tozzi}. Consider now a class $[\omega]\in \overline{H}^k_{2,\max}(M,g)$ and let $\omega$ be the unique representative of $[\omega]$ lying in $\mathcal{H}^k_{2,\mathrm{abs}}(M,g)$. Thanks to \eqref{thanks} we know that $\omega\in \ker(d_{k,\max,p})$ for each $2\leq p\leq \infty$. Thus we define $\beta_{p}([\omega])$ as the class induced by $\omega$, the unique harmonic representative of $[\omega]$, in  $\overline{H}^k_{p,\max}(M,g)$. Clearly  $\beta_p$ is well defined and linear. In order to verify that $\beta_p$ is injective,  assume that $[\omega]=0$ in $\overline{H}^k_{p,\max}(M,g)$, i.e., that $\omega \in \overline{\mathrm{Im}(d_{k-1,\max,2})}$, and let us consider a sequence $\{\phi_n\}_{n\in \mathbb{N}}\subset \mathcal{D}(d_{k-1,\max,p})\subset L^{p}\Omega^{k-1}(M,g)$ such that $d_{k-1,\max,p}\phi_n\rightarrow \omega$ in $L^{p}\Omega^{k}(M,g)$ as $n\rightarrow \infty$. Since $\vol_g(M)<\infty$ we have $\{\phi_n\}_{n\in \mathbb{N}}\subset \mathcal{D}(d_{k-1,\max,2})\subset L^{2}\Omega^{k-1}(M,g)$ and $d_{k-1,\max,2}\phi_n\rightarrow \omega$ in $L^{2}\Omega^{k}(M,g)$. This implies that $\omega=0$ as $\omega\in \mathcal{H}^k_{2,\mathrm{abs}}(M,g)\cap \overline{\mathrm{Im}(d_{k-1,\max,2})}$ and so we can conclude that $\beta_p$ is injective.  Concerning the map $\alpha_{\infty}$ let us consider again  a class $[\omega]\in \overline{H}^k_{2,\max}(M,g)$ and let $\omega$ be the unique representative of $[\omega]$ lying in $\mathcal{H}^k_{2,\mathrm{abs}}(M,g)$. Clearly $\omega$ is smooth and closed and thanks to Prop. \ref{tozzi} we know that $\omega\in L^{\infty}\Omega^k(M,g)$. Thus $\omega\in \ker(d_{k,\max,\infty})\cap \Omega^k(M)$. So we define $\alpha_{\infty}([\omega])$ as the class induced by $\omega$, the unique harmonic representative of $[\omega]$, in  $\im(H^k_{s,\infty}(M,g)\rightarrow\overline{H}^k_{\infty,\max}(M,g))$. It is immediate to check that $\alpha_{\infty}$ is well defined and linear. In order to verify that $\alpha_{\infty}$ is injective we first note that $$\im(H^k_{s,\infty}(M,g)\rightarrow\overline{H}^k_{\infty,\max}(M,g))=(\ker(d_{k,\max,\infty})\cap \Omega^k(M))/(\overline{\mathrm{Im}(d_{k-1,\max,\infty})}\cap \Omega^k_{s,\infty}(M)).$$ Likewise the previous case let us assume that there exists a sequence $\{\psi_n\}_{n\in \mathbb{N}}\subset \mathcal{D}(d_{k-1,\max,\infty})\subset L^{\infty}\Omega^{k-1}(M,g)$ such that $d_{k-1,\max,\infty}\psi_n\rightarrow \omega$ in $L^{\infty}\Omega^{k}(M,g)$ as $n\rightarrow \infty$. Since $\vol_g(M)<\infty$ we have $\{\psi_n\}_{n\in \mathbb{N}}\subset \mathcal{D}(d_{k-1,\max,2})\subset L^{2}\Omega^{k-1}(M,g)$ and $d_{k-1,\max,2}\psi_n\rightarrow \omega$ in $L^{2}\Omega^{k}(M,g)$. Therefore $\omega\in \mathcal{H}^k_{2,\mathrm{abs}}(M,g)\cap \overline{\mathrm{Im}(d_{k-1,\max,2})}$ and thus $\omega=0$ as $\omega\in \mathcal{H}^k_{2,\mathrm{abs}}(M,g)\cap \overline{\mathrm{Im}(d_{k-1,\max,2})}=\{0\}$. This shows that $\alpha_{\infty}$ is injective.
\end{proof}

\begin{rem}
\label{car}
 If $(M,g)$ is non-compact, complete and supports a Sobolev embedding as \eqref{imagine} then $\vol_g(M)$ is not finite, see \cite[Th. 4.1.2]{QZhang} This is why we required $(M,g)$ to be incomplete in the statement of the above theorem.
\end{rem}

Note that without any further hypothesis on $(M,g)$ the requirement $\Delta_{k,\mathrm{abs}}=\Delta_{k}^{\mathcal{F}}$ implies $d_{k,\max}=d_{k,\min}$, see Prop. \ref{oneto}, which in turn implies $H^{k}_{2,\max}(M,g)=\im(H^k_{2,\min}(M,g)\rightarrow H^k_{2,\max}(M,g))$ and $\overline{H}^{k}_{2,\max}(M,g)=\im(\overline{H}^k_{2,\min}(M,g)\rightarrow \overline{H}^k_{2,\max}(M,g))$ where on the right hand side of  the last two equalities we have the maps induced by the inclusion of domains $\mathcal{D}(d_{k,\min,2})\hookrightarrow \mathcal{D}(d_{k,\max,2})$.
In the next corollary we will see how  to strengthen the assumptions of Th. \ref{tozziti} in order to compare $\mathcal{H}^k_{2,\mathrm{abs}}(M,g)$ with  $\im(\overline{H}^k_{q,\min}(M,g)\rightarrow \overline{H}^k_{q,\max}(M,g))$.

\begin{cor}
In the setting of Th. \ref{tozziti}. Assume in addition that $(M,g)$ is $p$-parabolic for some $2< p<\infty$. Then for any $q\in (2,p]$ and  $\omega \in \mathcal{H}_{2,\mathrm{abs}}^k(M,g)$ we have  $\beta_q(\omega)\in \im(\overline{H}^k_{q,\min}(M,g)\rightarrow\overline{H}^k_{q,\max}(M,g))$ and the corresponding map $$\beta_q:\mathcal{H}^k_{2,\mathrm{abs}}(M,g)\rightarrow \im(\overline{H}^k_{q,\min}(M,g)\rightarrow\overline{H}^k_{q,\max}(M,g))$$ is injective.
\end{cor}

\begin{proof}
First we remark that $\im(\overline{H}^k_{q,\min}(M,g)\rightarrow\overline{H}^k_{q,\max}(M,g))=\ker(d_{k,\min,q})/(\overline{\mathrm{Im}(d_{k-1,\max,q})}\cap \ker(d_{k,\min,q}))$. Since $\vol_g(M)$ is finite we know that $(M,g)$ is $q$-parabolic for any $1\leq q\leq p.$ Thus by Prop. \ref{tozzi} and  \cite[Prop. 2.5]{Fra} we know that $\omega\in \ker(d_{k,\min,q})$ for any $1\leq q\leq p.$ Therefore $\beta_q(\omega)\in \im(\overline{H}^k_{q,\min}(M,g)\rightarrow\overline{H}^k_{q,\max}(M,g))$. Moreover by Th. \ref{tozziti} we already know that  if $\omega\in \overline{\mathrm{Im}(d_{k-1,\max,q})}$ then $\omega=0$, with $q\in [2,p]$. So we can conclude that $\beta_q:\mathcal{H}^k_{2,\mathrm{abs}}(M,g)\rightarrow \im(\overline{H}^k_{q,\min}(M,g)\rightarrow\overline{H}^k_{q,\max}(M,g))$ is injective as required.
\end{proof}

The next one is an obvious consequence of Th. \ref{tozziti}.

\begin{cor}
In the setting of Th. \ref{tozziti}. 
\begin{enumerate}
\item If $\dim(\im(H^k_{s,\infty}(M,g)\rightarrow\overline{H}^k_{\infty}(M,g))$ is finite dimensional then $\overline{H}^k_2(M,g)$ is finite dimensional too and $$\dim(\overline{H}^k_2(M,g))\leq \dim(\im(H^k_{s,\infty}(M,g)\rightarrow\overline{H}^k_{\infty}(M,g))).$$
\item If $p\geq 2$ and $\overline{H}^k_{p,\max}(M,g)$ is finite dimensional then $\overline{H}^k_2(M,g)$ is finite dimensional too and $$\dim(\overline{H}^k_{2}(M,g))\leq \dim(\overline{H}^k_{p,\max}(M,g)).$$
\item If $2\leq q\leq p$, $(M,g)$ is $p$-parabolic and $\dim(\im(\overline{H}^k_{q,\min}(M,g)\rightarrow\overline{H}^k_{q,\max}(M,g)))$ is finite dimensional then $\overline{H}^k_{2}(M,g)$ is finite dimensional too and $$\dim(\overline{H}^k_2(M,g))\leq\dim(\im(\overline{H}^k_{q,\min}(M,g)\rightarrow\overline{H}^k_{q,\max}(M,g))).$$
\end{enumerate}
\end{cor}

In the remaining part of this subsection we investigate a kind of converse of Th. \ref{tozziti}. More precisely assuming $\vol_g(M)<\infty$ we want to find some sufficient conditions that guarantee that the map $H^k_{q,\max}(M,g)\rightarrow H^k_2(M,g)$ induced by the inclusion $L^q\Omega^k(M,g)\hookrightarrow L^2\Omega^k(M,g)$ is injective.

\begin{teo}
\label{LpL2}
Let $(M,g)$ be an open incomplete Riemannian manifold of dimension $m>2$. Assume that
\begin{itemize}
\item $\vol_g(M)<\infty$;
\item There exists $k\in \{0,...,m\}$ such that $\ell_{k-1}^-,\ell_k^-\in \mathcal{K}(M)$;
\item $\mathrm{Im}(d_{k-1,\max,q})$ is closed in $L^q\Omega^k(M,g)$ for some $2\leq q< \infty$;
\item We have a continuous inclusion $W^{1,2}_0(M,g)\hookrightarrow L^{\frac{2m}{m-2}}(M,g)$;
\item  The operator $$\Delta_{k,\mathrm{abs}}:L^2\Omega^k(M,g)\rightarrow L^2\Omega^k(M,g)\quad\quad \mathrm{equals}\quad\quad \Delta_k^{\mathcal{F}}:L^2\Omega^k(M,g)\rightarrow L^2\Omega^k(M,g)$$
and analogously the operator $$\Delta_{k-1,\mathrm{abs}}:L^2\Omega^{k-1}(M,g)\rightarrow L^2\Omega^{k-1}(M,g)\quad\quad \mathrm{equals}\quad\quad \Delta_{k-1}^{\mathcal{F}}:L^2\Omega^{k-1}(M,g)\rightarrow L^2\Omega^{k-1}(M,g).$$
\end{itemize}
Then the map $$\gamma: H^k_{q,\max}(M,g)\rightarrow H^k_2(M,g)$$ induced by the continuous inclusion $L^q\Omega^k(M,g)\hookrightarrow L^2\Omega^k(M,g)$  is injective.
\end{teo}

Note that the equalities in the fifth item above implies $d_{k-1,\max,2}=d_{k-1,\min,2}$ and $d_{k,\max,2}=d_{k,\min,2}$. Hence there is a unique $L^2$-cohomology group at the level of $k$-forms that we denote with $H^k_2(M,g)$.

\begin{proof} 
We start with a remark about the notation. For any $1\leq p< \infty$ and $\omega\in \ker(d_{k,\max,p})$ with $[\omega]_p$ we denote the class induced by $\omega$ in $H^k_{p,\max}(M,g)$. Moreover given any $2\leq q<\infty$ and an arbitrary class $[\omega]_q\in H^k_{q,\max}(M,g)$ with $\omega\in \ker(d_{k,\max,q})$ a representative for $[\omega]_q$, we recall that  $\gamma([\omega]_q)$ is defined as the class induced by $\omega$ in $H^k_{2}(M,g)$. In other words $\gamma([\omega]_q):=[\omega]_2$. Let us now deal with the proof of the theorem. Assume that $\gamma([\omega]_q)=0$, that is $d_{k-1,2}\eta=\omega$ for some $\eta\in \mathcal{D}(d_{k-1,2})\subset L^2\Omega^{k-1}(M,g)$. Since $\overline{H}^k_{q,\max}(M,g)=H^k_{q,\max}(M,g)$ it is enough to show that $$\omega\in \overline{\im(d_{k-1,\max,q})}\subset L^q\Omega^k(M,g).$$  Let $\{t_j\}_{j\in \mathbb{N}}$ be a sequence such that $1>t_j>0$ and $t_j\rightarrow 0$ as $j\rightarrow \infty$. Let $\eta_j:=e^{-t_j\Delta_{k-1,\mathrm{abs}}}\eta$. We claim that 
$$\{\eta_j\}_{j\in \mathbb{N}}\subset L^q\Omega^{k-1}(M,g)\cap \Omega^{k-1}(M)\cap \mathcal{D}(d_{k-1,2}).$$ The inclusion $\{\eta_j\}_{j\in \mathbb{N}}\subset \mathcal{D}(d_{k-1,2})\cap \Omega^{k-1}(M)$ is clear and follows immediately by the properties of the heat operator. The remaining inclusion $\{\eta_j\}_{j\in \mathbb{N}}\subset L^q\Omega^{k-1}(M,g)$ follows by \eqref{klplq} as we assumed $\Delta_{k-1,\mathrm{abs}}=\Delta_{k-1}^{\mathcal{F}}$ and $\ell_{k-1}^-\in \mathcal{K}(M)$.  Now we want to show that $d_{k-1}\eta_j\in L^q\Omega^k(M,g)$. First we point out that $d_{k-1,2}e^{-t_j\Delta_{k-1,\mathrm{abs}}}\eta=e^{-t_j\Delta_{k,\mathrm{abs}}}d_{k-1,2}\eta$, see Prop. \ref{commu}. The latter equality clearly holds true in $L^2\Omega^k(M,g)$. On the other hand we have smooth $k$-forms on both sides of the equality. Therefore $d_{k-1,2}e^{-t_j\Delta_{k-1,\mathrm{abs}}}\eta=e^{-t_j\Delta_{k,\mathrm{abs}}}d_{k-1,2}\eta$ in $\Omega^k(M)$. 
In this way we are in position  to conclude  that $d_{k-1}\eta_j\in L^q\Omega^k(M,g)$ in that
\begin{align}
&\|d_{k-1}\eta_j\|_{L^q\Omega^k(M,g)}=\|d_{k-1}e^{-t_j\Delta_{k-1,\mathrm{abs}}}\eta\|_{L^q\Omega^k(M,g)}= \|d_{k-1,2}e^{-t_j\Delta_{k-1,\mathrm{abs}}}\eta\|_{L^q\Omega^k(M,g)}=\\
& \nonumber \|e^{-t_j\Delta_{k,\mathrm{abs}}}d_{k-1,2}\eta\|_{L^q\Omega^k(M,g)}= \|e^{-t_j\Delta_{k,\mathrm{abs}}}\omega\|_{L^q\Omega^k(M,g)}= \|e^{-t_j\Delta_{k}^{\mathcal{F}}}\omega\|_{L^q\Omega^k(M,g)}<\infty
\end{align}
with
$\|e^{-t_j\Delta_k^{\mathcal{F}}}\omega\|_{L^q\Omega^k(M,g)}<\infty$ thanks to \eqref{klqlq}. We can thus conclude that $\{\eta_j\}_{j\in \mathbb{N}}\subset \mathcal{D}(d_{k-1,\max,q})$ and $d_{k-1,\max,q}\eta_j=d_{k-1}\eta_j$  since $\eta_j$ is smooth. Furthermore we note that $d_{k-1,2}\eta_j=d_{k-1,\max,q}\eta_j$ for each $j\in\mathbb{N}$ as follows by \eqref{corno}.  As a next step consider the sequence $\{d_{k-1,\max,q}\eta_j\}_{j\in \mathbb{N}}$. We know that it is bounded in $L^q\Omega^k(M,g)$ as 
\begin{align}
& \nonumber \|d_{k-1,\max,q}\eta_j\|_{L^q\Omega^k(M,g)}=\|d_{k-1,2}e^{-t_j\Delta_{k-1,\mathrm{abs}}}\eta\|_{L^q\Omega^k(M,g)}=\|e^{-t_j\Delta_{k,\mathrm{abs}}}d_{k-1,2}\eta\|_{L^q\Omega^k(M,g)}=\\ 
& \nonumber \|e^{-t_j\Delta_{k}^{\mathcal{F}}}d_{k-1,2}\eta\|_{L^q\Omega^k(M,g)}\leq \delta e^{t_jc}\|d_{k-1,2}\eta\|_{L^q\Omega^k(M,g)}\leq\delta e^{c}\|\omega\|_{L^q\Omega^k(M,g)}
\end{align}
where in the second to last inequality above we have used \eqref{klqlq}. By the fact that  $L^q\Omega^k(M,g)$ is a reflexive Banach space we know that there exists a subsequence $\{t_u\}\subset \{t_j\}$ such that $\{d_{k-1,\max,q}\eta_u\}$ converges weakly to some $\zeta\in L^q\Omega^k(M,g)$, that is $d_{k-1,\max,q}\eta_u\rightharpoonup \zeta$ as $u\rightarrow \infty$. Since $\vol_g(M)<\infty$ we have a continuous inclusion $L^q\Omega^k(M,g)\hookrightarrow L^2\Omega^k(M,g)$ and therefore $d_{k-1,\max,q}\eta_u\rightharpoonup \zeta$ in $L^2\Omega^k(M,g)$ as $u\rightarrow \infty$, see \eqref{weak}. On the other hand using the fact that $e^{-t\Delta_{k,\mathrm{abs}}}$ is a strongly continuous semigroup of contractions, see \cite[Prop. 6.14]{KSC}, and keeping in mind that $d_{k-1,2}\eta_u=d_{k-1,\max,q}\eta_u$, we can conclude that $d_{k-1,2}\eta_u\rightarrow \omega$ in $L^2\Omega^k(M,g)$ in that
\begin{align}
& \nonumber \lim_{u\rightarrow \infty} \|d_{k-1,2}\eta_u-\omega\|_{L^2\Omega^k(M,g)}=\lim_{u\rightarrow \infty} \|d_{k-1,2}e^{-t_u\Delta_{k-1,\mathrm{abs}}}\eta-\omega\|_{L^2\Omega^k(M,g)}=\\
& \nonumber \lim_{u\rightarrow \infty} \|e^{-t_u\Delta_{k,\mathrm{abs}}}d_{k-1,2}\eta-\omega\|_{L^2\Omega^k(M,g)}=\lim_{u\rightarrow \infty} \|e^{-t_u\Delta_{k,\mathrm{abs}}}\omega-\omega\|_{L^2\Omega^k(M,g)}=0.
\end{align}
Thus we deduce that $\omega=\zeta$ in $L^2\Omega^k(M,g)$ and so $\omega=\zeta$ in $L^q\Omega^k(M,g)$. We are finally in position to show that $\omega\in \overline{\mathrm{\im}(d_{k-1,\max,q})}$. Thanks to \eqref{vid} it is enough to check that $$\int_Mg(\omega,\psi)\dvol_g=0$$ for each $\psi\in \ker(d^t_{k-1,\min,q'})$ with $q'=\frac{q}{q-1}$. We have 
$$ \nonumber \int_Mg(\omega,\psi)\dvol_g=\lim_{u\rightarrow\infty}\int_Mg(d_{k-1,\max,q}\eta_u,\psi)\dvol_g=\lim_{u\rightarrow\infty}\int_Mg(\eta_u,d^t_{k-1,\min,q'}\psi)\dvol_g=0.$$
We can thus conclude that $[\omega]_q=0$ in $\overline{H}^k_{q,\max}(M,g)=H^k_{q,\max}(M,g)$ as desired.  
\end{proof}

\begin{cor}
\label{IsoL2Lq}
In the setting of Th. \ref{LpL2}. If $\im(d_{k-1,\max,2})$ is closed in $L^2\Omega^k(M,g)$ then $\gamma:H^k_{q,\max}(M,g)\rightarrow H^k_2(M,g)$ is an isomorphism.
\end{cor}

\begin{proof}
We need to check that $\gamma:H^k_{q,\max}(M,g)\rightarrow H^k_2(M,g)$ is surjective. Let $[\omega]_2$ be an arbitrary class in $H^k_2(M,g)$. Since $\im(d_{k-1,\max,2})$ is closed we know that $H^k_2(M,g)=\overline{H}^k_2(M,g)$ and so there exists a unique representative $\omega$ of $[\omega]_2$ with $\omega\in \mathcal{H}^k_{\mathrm{abs}}(M,g)$. Thanks to  Th. \ref{tozziti} we know that $\omega$ induces a nontrivial class $[\omega]_q\in H^k_{q,\max}(M,g)$. Clearly $\gamma([\omega]_q)=[\omega]_2$ and  this leads us to conclude that $\gamma:H^k_{q,\max}(M,g)\rightarrow H^k_2(M,g)$ is  surjective and hence an isomorphism.
\end{proof}
%

\subsection{Heat operator and $L^p$-Stokes theorem}
The goal of this subsection is  to provide some sufficient conditions that imply the validity of the $L^q$-Stokes theorem. We start with the next lemma which relies on some ideas already used in the proof of Th. \ref{LpL2}.
\begin{lemma}
\label{lemma}
Let $(M,g)$ be an open and possibly incomplete Riemannian manifold of finite volume such that 
\begin{itemize}
\item There exists $k\in \{0,...,m\}$ such that $\ell_k^-,\ell_{k+1}^-\in \mathcal{K}(M)$;
\item $\Delta_{k,\mathrm{abs}}=\Delta_{k}^{\mathcal{F}}$ and $\Delta_{k+1,\mathrm{abs}}=\Delta_{k+1}^{\mathcal{F}}$ as unbounded and self-adjoint operators acting on $L^2\Omega^k(M,g)$ and $L^2\Omega^{k+1}(M,g)$, respectively. 
\end{itemize}
Given $2\leq q< \infty$ let $$\mathcal{C}_q:=\{e^{-t\Delta_{k,\mathrm{abs}}}\eta:\ \eta\in \mathcal{D}(d_{k,\max,q}),\ t\in (0,1]\}.$$ Then for every $\eta\in \mathcal{D}(d_{k,\max,q})$ there exists a sequence $\{\eta_j\}_{j\in \mathbb{N}}\subset \mathcal{C}_q$ such that as $j\rightarrow \infty$ $$\eta_j\rightharpoonup \eta\quad\quad  \mathrm{and}\quad\quad d_{k,\max,q}\eta_j\rightharpoonup d_{k,\max,q}\eta$$ in $L^q\Omega^k(M,g)$ and $L^q\Omega^{k+1}(M,g)$, respectively.
\end{lemma}

\begin{proof}
Our first aim is to show that  $\mathcal{C}_q\subset \mathcal{D}(d_{k,\max,q})$. Let $\eta\in \mathcal{D}(d_{k,\max,q})$ be arbitrarily fixed. Since $\eta\in L^q\Omega^k(M,g)\hookrightarrow L^2\Omega^k(M,g)$ and $e^{-t\Delta_{k,\mathrm{abs}}}\eta$ is smooth we have $$d_{k}e^{-t\Delta_{k,\mathrm{abs}}}\eta=d_{k,2}e^{-t\Delta_{k,\mathrm{abs}}}\eta=e^{-t\Delta_{k+1,\mathrm{abs}}}d_{k,2}\eta=e^{-t\Delta_{k+1,\mathrm{abs}}}d_{k,\max,q}\eta$$ and thus for the $L^q$ norms we have 
$$\|e^{-t\Delta_{k,\mathrm{abs}}}\eta\|_{L^q\Omega^k(M,g)}\leq \delta e^{tc}\|\eta\|_{L^q\Omega^k(M,g)}\leq \delta e^{c}\|\eta\|_{L^q\Omega^k(M,g)}$$ 
and
\begin{align}
 \nonumber  \|d_{k}e^{-t\Delta_{k,\mathrm{abs}}}\eta\|_{L^q\Omega^k(M,g)}&=\|e^{-t\Delta_{k+1,\mathrm{abs}}}d_{k,\max,q}\eta\|_{L^q\Omega^k(M,g)}\leq\\
& \nonumber \delta e^{tc}\|d_{k,\max,q}\eta\|_{L^q\Omega^{k+1}(M,g)}\leq \delta e^{c}\|d_{k,\max,q}\eta\|_{L^q\Omega^{k+1}(M,g)}
\end{align}
where in both lines the last inequality follows by \eqref{klqlq}. This shows that $\mathcal{C}_q\subset \mathcal{D}(d_{k,\max,q})$. Consider now a sequence $\{t_j\}_{j\in \mathbb{N}}\subset (0,1]$ with $t_j\rightarrow 0$ as $j\rightarrow \infty$. Given any $\eta\in \mathcal{D}(d_{k,\max,q})$ let $\eta_j:=e^{-t_j\Delta_{k,\mathrm{abs}}}\eta$. Then $$\|\eta_j\|_{L^q\Omega^k(M,g)}\leq \delta e^{c}\|\eta\|_{L^q\Omega^k(M,g)}\quad\quad\ \mathrm{and}\ \quad\quad\|d_k\eta_j\|_{L^q\Omega^{k+1}(M,g)}\leq \delta e^{c}\|d_{k,\max,q}\eta\|_{L^q\Omega^{k+1}(M,g)}.$$ Therefore both sequences are bounded in $L^q\Omega^{k}(M,g)$ and $L^q\Omega^{k+1}(M,g)$, respectively, and since these are reflexive Banach spaces we get that $\eta_j\rightharpoonup \alpha$ to some  $\alpha \in L^q\Omega^k(M,g)$ and $d_{k,\max,q}\eta_j\rightharpoonup \beta$ to some $\beta\in L^q\Omega^{k+1}(M,g)$,  as $j\rightarrow \infty$. On the other hand we know that $\{\eta_j\}_{j\in\mathbb{N}}\subset \mathcal{D}(d_{k,\max,2})$ and using the fact that both $e^{-t\Delta_{k,\mathrm{abs}}}$ and $e^{-t\Delta_{k+1,\mathrm{abs}}}$ are strongly continuous with respect to $t$, see e.g. \cite[Prop. 6.14]{KSC}, we have $$\lim_{j\rightarrow \infty}\|\eta-\eta_j\|_{L^2\Omega^k(M,g)}=\lim_{j\rightarrow \infty}\|\eta-e^{-t_j\Delta_{k,\mathrm{abs}}}\eta\|_{L^2\Omega^k(M,g)}=0$$ and 
\begin{align}
& \nonumber \lim_{j\rightarrow \infty}\|d_{k,\max,2}\eta-d_{k,\max,2}\eta_j\|_{L^2\Omega^{k+1}(M,g)}=\lim_{j\rightarrow \infty}\|d_{k,\max,2}\eta-d_{k,\max,2}e^{-t_j\Delta_{k,\mathrm{abs}}}\eta\|_{L^2\Omega^{k+1}(M,g)}=\\
& \nonumber \lim_{j\rightarrow \infty}\|d_{k,\max,2}\eta-e^{-t_j\Delta_{k+1,\mathrm{abs}}}d_{k,\max,2}\eta\|_{L^2\Omega^{k+1}(M,g)}=0.
\end{align}
In such a way, thanks to \eqref{weak},  we can conclude that $\alpha=\eta$ in $L^2\Omega^k(M,g)$ and $\beta=d_{k,\max,q}\eta$ in $L^2\Omega^{k+1}(M,g)$ and therefore $\alpha=\eta$ in $L^q\Omega^k(M,g)$ and $\beta=d_{k,\max,q}\eta$ in $L^q\Omega^{k+1}(M,g)$ as desired.
\end{proof}

\begin{teo} 
\label{LpStokes}
Let $(M,g)$ be an open and  incomplete Riemannian manifold of dimension $m>2$ such that 
\begin{itemize}
\item $\vol_g(M)<\infty$;
\item There is a continuous inclusion $W^{1,2}_0(M,g)\hookrightarrow L^{\frac{2m}{m-2}}(M,g)$;
\item $(M,g)$ is $q$-parabolic for some $2<q<\infty$;
\item There exists $k\in \{0,...,m\}$ such that $\ell_k^-,\ell_{k+1}^-\in \mathcal{K}(M)$;
\item $\Delta_{k,\mathrm{abs}}=\Delta_{k}^{\mathcal{F}}$ and $\Delta_{k+1,\mathrm{abs}}=\Delta_{k+1}^{\mathcal{F}}$ as unbounded and self-adjoint operators acting on $L^2\Omega^k(M,g)$ and $L^2\Omega^{k+1}(M,g)$, respectively. 
\end{itemize}
Then the $L^r$-Stokes theorem holds true on $L^r\Omega^{k}(M,g)$ for each $r\in [2,q]$.
\end{teo}

\begin{proof}
By the fact that $\vol_g(M)<\infty$ and $(M,g)$ is $q$-parabolic we know that $(M,g)$ is $r$-parabolic for any $r\in [1,q]$. We will prove the theorem only for the case $r=q$. The remaining cases $2\leq r<q$ follow in the same manner. As in the previous lemma let  $$\mathcal{C}_q:=\{e^{-t\Delta_{k,\mathrm{abs}}}\eta:\  \eta\in \mathcal{D}(d_{k,\max,q}),\ t\in (0,1]\}.$$ As a first step we want to show that  $\mathcal{C}_q\subset \mathcal{D}(d_{k,\min,q})$.  As explained in the proof of Lemma \ref{lemma} we know that $\mathcal{C}_q\subset \mathcal{D}(d_{k,\max,q})$. Moreover the assumptions allow us to use  \eqref{klplq} to conclude that  $\mathcal{C}_q\subset L^{\infty}\Omega^k(M,g)$.  Thus we know that $\mathcal{C}_q\subset \mathcal{D}(d_{k,\max,q})\cap L^{\infty}\Omega^k(M,g)$ and so, by using the fact $(M,g)$ is $q$-parabolic, we are in position to use \cite[Prop. 2.5]{Fra} to conclude that  $\mathcal{C}_q\subset \mathcal{D}(d_{k,\min,q})$. Let us tackle now the statement of this theorem. Thanks to \eqref{naghen} it amounts to proving that given any $\eta\in \mathcal{D}(d_{k,\max,q})$ we have $$\int_Mg(\eta,d^t_{k,\max,q'}\omega)\dvol_g=\int_Mg(d_{k,\max,q}\eta,\omega)\dvol_g$$ for each $\omega\in \mathcal{D}(d^t_{k,\max,q'})$ and $q'=\frac{q}{q-1}$.
So let $\eta\in \mathcal{D}(d_{k,\max,q})$ and $\omega\in \mathcal{D}(d^t_{k,\max,q'})$  be  arbitrarily fixed. Thanks to Lemma \ref{lemma}, the inclusion $\mathcal{C}_q\subset \mathcal{D}(d_{k,\min,q})$ and \eqref{naghen} we have
\begin{align}
&  \nonumber \int_Mg(\eta,d^t_{k,\max,q'}\omega)\dvol_g=\lim_{j\rightarrow \infty}\int_Mg(\eta_j,d^t_{k,\max,q'}\omega)\dvol_g=\\
& \nonumber \lim_{j\rightarrow \infty}\int_Mg(d_{k,\min,q}\eta_j,\omega)\dvol_g=\lim_{j\rightarrow \infty}\int_Mg(d_{k,\max,q}\eta_j,\omega)\dvol_g=\int_Mg(d_{k,\max,q}\eta,\omega)\dvol_g
\end{align}
where $\{\eta_j\}_{j\in \mathbb{N}}\subset \mathcal{C}_q$ is any sequence defined as in Lemma \ref{lemma}. We can thus conclude that $$d_{k,\max,q}:L^q\Omega^k(M,g)\rightarrow L^q\Omega^{k+1}(M,g)\quad\quad \mathrm{equals}\quad\quad d_{k,\min,q}:L^q\Omega^k(M,g)\rightarrow L^q\Omega^{k+1}(M,g)$$ as desired.
\end{proof}

\begin{cor}
\label{againLpStokes}
In the setting of Th. \ref{LpStokes}. We have the following equalities: $$\overline{H}^k_{r,\max}(M,g)=\overline{H}^k_{r,\min}(M,g)\quad\quad \mathrm{and}\quad\quad H^k_{r,\max}(M,g)=H^k_{r,\min}(M,g)$$ with $2\leq r\leq q$. If in addition $M$ is oriented then the $L^r$-Stokes theorem holds true on $m-k-1$ forms for each $r\in [\frac{q}{q-1},2]$ and consequently $$\overline{H}^{m-k-1}_{r,\max}(M,g)=\overline{H}^{m-k-1}_{r,\min}(M,g)\quad\quad \mathrm{and}\quad\quad H^{m-k-1}_{r,\max}(M,g)=H^{m-k-1}_{r,\min}(M,g)$$ with $\frac{q}{q-1}\leq r\leq 2$.
\end{cor}

\begin{proof}
The first statement follows at once by Th. \ref{LpStokes}. Consider now the second assertion. Since $r\in [\frac{q}{q-1},2]$ we know that $r'=\frac{r}{r-1}\in [2,q])$ and so by applying Th. \ref{LpStokes} we know that $d_{k,\max,r'}:L^{r'}\Omega^k(M,g)\rightarrow L^{r'}\Omega^{k+1}(M,g)$ equals $d_{k,\min,r'}:L^{r'}\Omega^k(M,g)\rightarrow L^{r'}\Omega^{k+1}(M,g)$. Clearly this implies immediately that $d^t_{k,\max,r}:L^{r}\Omega^{k+1}(M,g)\rightarrow L^{r}\Omega^{k}(M,g)$ equals $d^t_{k,\min,r}:L^{r}\Omega^{k+1}(M,g)\rightarrow L^{r}\Omega^{k}(M,g)$. Since $M$ is oriented we have $d^t_k=(-1)^{mk+1}*d_{m-k-1}*$ which in turn is easily seen to imply $d^t_{k,\min/\max,r}=(-1)^{mk+1}*d_{m-k-1,\min/\max,r}*$. We can thus conclude that $d_{m-k-1,\max,r}:L^{r}\Omega^{m-k-1}(M,g)\rightarrow L^{r}\Omega^{m-k}(M,g)$ equals $d_{m-k-1,\min,r}:L^{r}\Omega^{m-k-1}(M,g)\rightarrow L^{r}\Omega^{m-k}(M,g)$ as desired. Finally the last assertion follows now immediately from the $L^r$-Stokes theorem. 
\end{proof}

\section{Geometric and topological applications}
This last section gathers various applications of the previous results to Thom-Mather stratified spaces and complex projective varieties with isolated singularities. Due to the extent of the subject we are forced to recall only what is strictly necessary for our aims. We start with a brief  introduction to smoothly Thom-Mather stratified pseudomanifolds and intersection cohomology.

\subsection{Background on Thom-Mather stratified spaces}
Since it will be used in the definition below we start by recalling that, given a topological space $Z$, $C(Z)$ stands for the cone over $Z$ that is, $C(Z)=Z\times [0,2)/\sim$ where $(p,t)\sim (q,r)$ if and only if $r=t=0$. We have now the following

\begin{defi}   
\label{thom}
 A smoothly Thom-Mather stratified pseudomanifold $X$  of dimension $m$  is a metrizable, locally compact, second countable space which admits a locally finite decomposition into a union of locally closed strata $\mathfrak{G}=\{Y_{\alpha}\}$, where each $Y_{\alpha}$ is a smooth, open and connected manifold, with dimension depending on the index $\alpha$. We assume the following:
\begin{enumerate}
\item[(i)] If $Y_{\alpha}$, $Y_{\beta} \in \mathfrak{G}$ and $Y_{\alpha} \cap \overline{Y}_{\beta} \neq \emptyset$ then $Y_{\alpha} \subset \overline{Y}_{\beta}$
\item[(ii)]  Each stratum $Y$ is endowed with a set of control data $T_{Y} , \pi_{Y}$ and $\rho_{Y}$ ; here $T_{Y}$ is a neighborhood of $Y$ in $X$ which retracts onto $Y$, $\pi_{Y} : T_{Y} \rightarrow Y$
is a fixed continuous retraction and $\rho_{Y}: T_{Y}\rightarrow [0, 2)$  is a continuous  function in this tubular neighborhood such that $\rho_{Y}^{-1}(0) = Y$. Furthermore, we require that if $Z \in \mathfrak{G}$ and $Z \cap T_{Y}\neq \emptyset$  then $(\pi_{Y} , \rho_{Y} ) : T_{Y} \cap Z \rightarrow Y\times [0,2) $ is a proper smooth submersion.
\item[(iii)] If $W, Y,Z \in \mathfrak{G}$, and if $p \in T_{Y} \cap T_{Z} \cap W$ and $\pi_{Z}(p) \in T_{Y} \cap Z$ then
$\pi_{Y} (\pi_{Z}(p)) = \pi_{Y} (p)$ and $\rho_{Y} (\pi_{Z}(p)) = \rho_{Y} (p)$.
\item[(iv)] If $Y,Z \in \mathfrak{G}$, then
$Y \cap \overline{Z} \neq \emptyset \Leftrightarrow T_{Y} \cap Z \neq \emptyset$ ,
$T_{Y} \cap T_{Z} \neq \emptyset \Leftrightarrow Y\subset \overline{Z}, Y = Z\ or\ Z\subset \overline{Y} .$
\item[(v)]  For each $Y \in \mathfrak {G}$, the restriction $\pi_{Y} : T_{Y}\rightarrow Y$ is a locally trivial fibration with fibre the cone $C(L_{Y})$ over some other stratified space $L_{Y}$ (called the link over $Y$ ), with atlas $\mathcal{U}_{Y} = \{(\phi,\mathcal{U})\}$ where each $\phi$ is a trivialization
$\pi^{-1}_{Y} (U) \rightarrow U \times C(L_{Y} )$, and the transition functions are stratified isomorphisms  which preserve the rays of each conic
fibre as well as the radial variable $\rho_{Y}$ itself, hence are suspensions of isomorphisms of each link $L_{Y}$ which vary smoothly with the variable $y\in U$.
\item[(vi)] For each $j$ let $X_{j}$ be the union of all strata of dimension less or equal than $j$, then 
$$
X_{m-1}=X_{m-2}\ \text{and  $X\setminus X_{m-2}$ dense\ in $X$}
$$
\end{enumerate}
\end{defi}

The \emph{depth} of a stratum $Y$ is largest integer $k$ such that there exists a chain of strata $Y=Y_{k},...,Y_{0}$ such that $Y_{j}\subset \overline{Y_{j-1}}$ for $1\leq j\leq k.$ A stratum of maximal depth is always a closed subset of $X$.  The  maximal depth of any stratum in $X$ is called the \emph{depth of $X$} as stratified spaces. Note that if $\mathrm{depth}(X)=1$ then for any singular stratum $Y\in \mathfrak{G}$ the corresponding link $L_Y$ is a smooth manifold. Consider now the filtration
\begin{equation}
X = X_{m}\supset X_{m-1}= X_{m-2}\supset X_{m-3}\supset...\supset X_{0}.
\label{pippo}
\end{equation}
 We refer to the open subset $X\setminus X_{m-2}$ as the regular set of $X$ while the union of all other strata is the singular set of $X$,
$$\reg(X):=X\setminus \sing(X)\quad \text{with}\quad \sing(X):=\bigcup_{Y\in \mathfrak{G}, \depth(Y)>0 }Y. $$
The dimension of $X$ is by definition the dimension of $\reg(X)$. Given two smoothly Thom-Mather  stratified pseudomanifolds  $X$ and $X'$,  a stratified isomorphism between them is a
homeomorphism $F: X\rightarrow X'$ which carries the open strata of $X$ to the open strata of $X'$
diffeomorphically, and such that $\pi'_{F(Y )}\circ  F = F \circ \pi_Y$ , $\rho'_{F(Y)}\circ F=\rho_Y$  for all $Y\in \mathfrak{G}(X)$.
For more details, properties and comments we refer to \cite{ALMP}, \cite{BHS}, \cite{RrHS},  \cite{JMA} and \cite{ver}. 
We recall in addition that important examples of smoothly Thom-Mather stratified pseudomanifolds are provided by complex projective varieties  or  quotient spaces $M/G$ without one codimensional strata, with $M$ a manifold and $G$ a Lie group acting properly on $M$.\\ 
As a next step we introduce the class of smooth Riemmanian metrics on $\reg(X)$  we will work with. The definition is given by induction on the depth of $X$.  
We recall first that  two Riemannian metrics $g$ and $h$ on a manifold $M$  are said to be \emph{quasi-isometric}, briefly $g\sim h$, if there exists a real number $c>0$ such that $c^{-1}h\leq g\leq ch$.

\begin{defi}
\label{iter}
Let $X$ be a smoothly Thom-Mather stratified pseudomanifold and let $g$ be a Riemannian metric on $\reg(X)$. If $\depth(X)=0$ that is, $X$ is a smooth manifold, an iterated  conic  metric on $X$ is  any smooth Riemannian metric on $X$. Suppose now that $\depth(X)=k$ and that the definition of  iterated  conic metric is given in the case $\depth(X)\leq k-1$; then we call a smooth Riemannian metric $g$ on $\reg(X)$ an iterated  conic metric   if it satisfies the following properties:
\begin{itemize}
\item Let $Y$ be a stratum of $X$ such that $Y\subset X_{i}\setminus  X_{i-1}$; by   definition \ref{thom} for each $q\in Y$ there exist an open neighbourhood $U$ of $q$ in $Y$ such that 
$$
\phi:\pi_{Y}^{-1}(U)\longrightarrow U\times C(L_{Y})
$$
 is a stratified isomorphism; in particular, 
$$
\phi:\pi_{Y}^{-1}(U)\cap \reg(X)\longrightarrow U\times \reg(C(L_{Y}))
$$
is a smooth diffeomorphism. Then, for each $q\in Y$, there exists one of these trivializations $(\phi,U)$ such that $g$ restricted on $\pi_{Y}^{-1}(U)\cap \reg(X)$ satisfies the following properties:
\begin{equation} 
(\phi^{-1})^{*}(g|_{\pi_{Y}^{-1}(U)\cap \reg(X)})\sim  dr^2+h_{U}+r^2g_{L_{Y}}
\label{yhn}
\end{equation}
 where  $h_{U}$ is a Riemannian metric defined over $U$, $g_{L_{Y}}$ is an iterated conic metric  on $\reg(L_{Y})$, $dr^2+h_{U}+r^2g_{L_{Y}}$ is a Riemannian metric of product type on $U\times \reg(C(L_{Y}))$ and with $\sim$ we mean \emph{quasi-isometric}. 
\end{itemize}
\label{edge}
\end{defi} 
Sometimes these metrics are called incomplete iterated edge metrics \cite{ALMP}  or wedge metrics \cite{ALMP2}.
It is easy to check that $(\reg(X),g)$ is an incomplete Riemannian manifold and that $\vol_g(\reg(X))<\infty$ in case $X$ is compact. A non obvious property is that iterated conic metrics do exist. We refer to  \cite{ALMP} or \cite{BHS} for existence results.  As we will recall later the importance of this class of metrics stems from its deep connection with the topology of $X$. In particular the $L^p$-cohomology groups of $\reg(X)$ with respect  to an iterated conic metric are isomorphic to certain intersection cohomology  grups of $X$.


\subsection{Intersection cohomology and $L^p$-cohomology}
In the first part of this subsection we recall the definition of intersection cohomology and the corresponding Poincar\'e duality.
We will be very brief and we refer to the seminal papers of  Goresky  and  MacPherson \cite{GoMac}-\cite{GoMac2} and to the  monographs \cite{Bana}, \cite{Boreal} and \cite{KiWoo} for  in-depth treatments of this subject. For the sake of simplicity we will state everything in the setting of smoothly Thom-Mather stratified pseudomanifolds but we warn the reader that the same results hold true in the more general class of topologically stratified pseudomanifolds. As a first step we need to recall the notion of perversity.

\begin{defi}
\label{perv}
 A perversity is a function $p:\{2,3,4,...,n\}\rightarrow \mathbb{N}$ such that
\begin{equation} p(2)=0\ and\ p(k)\leq p(k+1)\leq p(k)+1.
\end{equation}
\label{per}
\end{defi}

\noindent Example of perversities are the zero perversity $0(k)=0$, the top perversity $t(k)=k-2$, the upper middle perversity $\overline{m}(k)=[\frac{k-1}{2}]$ and the lower middle perversity $\underline{m}(k)=[\frac{k-2}{2}]$, with $[\bullet]$ denoting the integer part of $\bullet$. Note that $\underline{m}(k)=\overline{m}(k)$ when $k$ is even. Consider now the standard $i$-simplex $\Delta_{i}\subset \mathbb{R}^{i+1}$.  The $j$-skeleton of $\Delta_{i}$ is the set of $j$-subsimplices. Given a smoothly Thom-Mather stratified pseudomanifold $X$ we say that a singular $i$-simplex in $X$, i.e. a continuous map $\sigma:\Delta_{i}\rightarrow X$, is $p$-allowable if
$$ 
\sigma^{-1}(X_{m-k}-X_{m-k-1})\subset \{(i-k+p(k))-\mathrm{skeleton\ of}\ \Delta_{i}\}\ \mathrm{for\ all}\ k\geq 2.
$$
Note that to each singular stratum $Y\subset X$ $p$ assigns the value corresponding to $\mathrm{cod}(Y)$. Given a field $F$, the elements of the space $I^{p}S_{i}(X,F)$ are defined as the finite linear combinations with coefficients in $F$ of singular $i$-simplex $\sigma:\Delta_{i}\rightarrow X$ such that $\sigma$ and $\partial \sigma$ are both $p$-allowable. It is clear that $(I^{p}S_{i}(X,F), \partial_{i})$ is a complex and the  singular intersection homology groups with respect to the perversity $p$, $I^{p}H_{i}(X,F)$, are defined as  the homology groups of this complex. One of the main success of the intersection homology is that it allows to restore a kind of Poincar\'e duality on singular spaces:

\begin{teo}
\label{duality}
Let $F$ be a field, $X$ a compact and $F$-oriented smoothly Thom-Mather stratified pseudomanifold of dimension $n$ and $p$, $q$  general perversities on $X$ such that $p+q=t$. Then the following isomorphism holds:
\begin{equation}
I^{p}H_{i}(X,F)\cong \mathrm{Hom}(I^{q}H_{n-i}(X,F), F).
\end{equation}
\label{mzb}
\end{teo}
\begin{proof}
A proof can be found in \cite{GoMac2}, \cite{Bana}, \cite{KiWoo}.
\end{proof}

We remark that if we denote with $I^{p}H^{i}(X,F)$  the cohomology of the complex $(\mathrm{Hom}(I^{p}S_{i}(X,F), F), (\partial_{i})^{*})$  then, using the fact that $F$ is a field, we have $I^{p}H^{i}(X,F)\cong \mathrm{Hom}(I^{p}H_{i}(X,F), F).$  This concludes our very succinct summary on intersection cohomology. Now we recall some important results that relate analytic data on $\reg(X)$, such as the $L^p$-de Rham cohomology and the Hodge-de Rham operator $d+d^t$, with the intersection cohomology of $X$. Let $X$ be a compact smoothly  Thom-Mather stratified pseudomanifold of dimension $m$ and let $g$ be an iterated conic metric on $\reg(X)$. Given $2\leq r<\infty$ we define the following perversities: $$p_{r}(k):=[[k/r]]\quad\quad \mathrm{and}\quad\quad q_r=t-p_r$$ with $[[\bullet]]$ denoting the greatest integer number strictly smaller than $\bullet$ and $t$ the top perversity. It is an easy exercise to check that $p_r$ and $q_r$ are well defined, that is the properties required in Def. \ref{perv} are fulfilled. Moreover for $r=2$ we have $p_r=\overline{m}$, $q_r=\underline{m}$ while if we define $\ell:=\max\{\mathrm{cod}(Y): Y\in\mathfrak{G}, \depth(Y)>0\}$ then for any $r\geq \ell$ we have $p_r=0$ and $q_r=t$. We have now the following

\begin{teo}
\label{LpInt}
Let $X$ be a compact smoothly oriented Thom-Mather stratified pseudomanifold of dimension $m$ and let $g$ be an iterated conic metric on $\reg(X)$. We have the following isomorphisms for each $k=0,...,m$ and $2\leq r<\infty$:
$$H^k_{r,\max}(\reg(X),g)\cong I^{p_r}H_{m-k}(X,\mathbb{R})\cong I^{q_r}H^k(X,\mathbb{R}).$$
\end{teo}
\begin{proof}
This first isomorphism is due to Cheeger and Nagase for $r=2$ and to Youssin for $2\leq r<\infty$, see \cite{JC}, \cite{Masa} and \cite{BYOU}, respectively. The second isomorphism  follows by Th. \ref{duality}.
\end{proof}

\begin{cor}
\label{closed}
In the setting of theorem \ref{LpInt}. For any $k=0,...,m$ and $2\leq r<\infty$ we have the following properties
\begin{enumerate}
\item $H^k_{r,\max}(\reg(X),g)$ is finite dimensional;
\item $\im(d_{k,\max,r})$ is closed in $L^r\Omega^{k+1}(\reg(X),g)$ and so $\overline{H}^k_{r,\max}(\reg(X),g)=H^k_{r,\max}(\reg(X),g)$.
\end{enumerate}
\end{cor}

\begin{proof}
The first property above follows by the fact that $X$ is compact and thus $I^{q_r}H^k(X,\mathbb{R})$ is finite dimensional. The second property follow by the first one, see \cite[pag. 578]{BYOU}.
\end{proof}

\begin{cor}
\label{DLpInt}
In the setting of Th. \ref{LpInt}. For any $k=0,...,m$ and $2\leq r<\infty$ we have the following isomorphisms
$$H^k_{r',\min}(\reg(X),g)\cong I^{p_r}H^k(X,\mathbb{R})\cong I^{q_r}H_{m-k}(X,\mathbb{R}).$$ with $r'=\frac{r}{r-1}$.
Moreover  $H^k_{r',\min}(\reg(X),g)$ is finite dimensional, $\im(d_{k,\min,r'})$ is closed in $L^{r'}\Omega^{k+1}(\reg(X),g)$ and  $\overline{H}^k_{r',\min}(\reg(X),g)=H^k_{r',\min}(\reg(X),g)$.
\end{cor}

\begin{proof}
The isomorphisms  are a consequence of Th. \ref{LpInt}, Cor. \ref{closed} and Poincar\'e duality. More precisely thanks to Th. \ref{LpInt} and Cor. \ref{closed} we know that $H^k_{r,\max}(\reg(X),g)$ is finite dimensional and thus $\im(d_{k-1,\max,r})$ is closed in $L^r\Omega^{k}(\reg(X),g)$ for each $k=0,...,m$. As recalled in the introduction this in turn implies that $\im(d_{m-k,\min,r'})$ is closed in $L^{r'}\Omega^{m-k+1}(\reg(X),g)$ for each $k=0,...,m$. Now Poincar\'e duality for the $L^r$-cohomology \eqref{LpPoinca} tells us that $H^{m-k}_{r',\min}(\reg(X),g)\cong (H^{k}_{r,\max}(\reg(X),g))^*$. Eventually we proved that $H^{m-k}_{r',\min}(\reg(X),g)\cong (H^{k}_{r,\max}(\reg(X),g))^*$ $\cong (I^{p_r}H_{m-k}(X,\mathbb{R}))^*$ $\cong I^{p_r}H^{m-k}(X,\mathbb{R})$ $\cong I^{q_r}H_k(X,\mathbb{R})$. Equivalently  $H^k_{r',\min}(\reg(X),g)\cong I^{p_r}H^k(X,\mathbb{R})\cong I^{q_r}H_{m-k}(X,\mathbb{R})$.  
\end{proof}

We have now the following 

\begin{defi}
\label{Witt-Normal}
Let $X$ be a smoothly Thom-Mather stratified pseudomanifold. We say that $X$ is Witt if for any $Y\in \frak{G}$ of odd codimension, i.e. $\mathrm{cod}(Y)=2l_Y+1$, we have $I^{\underline{m}}H^{l_Y}(L_Y,\mathbb{R})=0$. We say that $X$ is normal if for any $Y\in \frak{G}$ the corresponding link $L_Y$ is connected.
\end{defi}

In the next theorem we collect important properties of Witt/normal smoothly Thom-Mather stratified pseudomanifold. In particular we will see that Witt spaces enjoy a true Poincar\'e duality:

\begin{teo}
\label{mischiozzo}
Let $F$ be a field and let $X$ be a compact and $F$-oriented smoothly Thom-Mather stratified pseudomanifold. 
\begin{itemize}
\item If $X$ is Witt then $I^{\underline{m}}H_{k}(X,F)\cong \mathrm{Hom}(I^{\underline{m}}H_{n-k}(X,F), F)$.
\item If $X$ is normal then $I^tH_k(X,\mathbb{R})\cong H_k(X,\mathbb{R})$ and $I^0H_k(X,\mathbb{R})\cong H^{m-k}(X,\mathbb{R})$.
\end{itemize}
\end{teo}
\begin{proof}
For the first property see \cite{GoMac2}, \cite{Bana}, \cite{KiWoo}. For the second one we refer to \cite[pag. 153]{GoMac}. 
\end{proof}

Finally we come to the last result of this section. Let $$d+d^t:\Omega_c^{\bullet}(\reg(X))\rightarrow \Omega_c^{\bullet}(\reg(X))$$ be the Hodge-de Rham operator with $d=\bigoplus_{k=0}^m d_k$, $d^t=\bigoplus_{k=0}^m d^t_k$ and $\Omega_c^{\bullet}(\reg(X))=\bigoplus_{k=0}^m \Omega^k_c(\reg(X))$:
\begin{teo}
Let $X$ be a compact, smoothly, Thom-Mather-Witt stratified pseudomanifold. Then there exist  iterated conic metrics $g$ on $\reg(X)$
such that $$d+d^t:L^2\Omega^{\bullet}(\reg(X),g)\rightarrow L^2\Omega^{\bullet}(\reg(X),g)$$ is essentially self-adjoint.
\end{teo}
\begin{proof}
See \cite[Th. 6.1, Prop. 5.4]{ALMP}. 
\end{proof}

\subsection{Applications to Thom-Mather pseudomanifolds}
From now on we will always assume that $\reg(X)$ is connected. Roughly speaking the main goal of this subsection is to show how in many cases  the existence of an iterated conic metric on $\reg(X)$ with the negative part of $L_k$, the linear term appearing in the Weitzenb\"ock formula for the k-th Hodge Laplacian $\Delta_k$, in the Kato class of $\reg(X)$ yields already strong topological constraints. Since they will be used frequently, we recall here two important properties enjoyed by compact  smoothly  Thom-Mather stratified pseudomanifolds of dimension $m$ whose regular part is endowed with an iterated conic metric:
\begin{itemize}
\item $(\reg(X), g)$ is parabolic;
\item If $m>2$ then $f\in W^{1,2}_0(\reg(X),g)$ implies $f\in L^{\frac{2m}{m-2}}(\reg(X),g)$ and the corresponding inclusion $W^{1,2}_0(\reg(X),g)\hookrightarrow L^{\frac{2m}{m-2}}(\reg(X),g)$ is continuous.
\end{itemize}
For the former property we refer to \cite[Th. 3.4]{BeGu} while the latter is proved in \cite[Prop. 2.2]{ACM}.
 We are now in position for the following

\begin{prop}
Let $X$ be a compact and oriented smoothly Thom-Mather stratified pseudomanifold of dimension $m$ such that $\ell_k^-\in \mathcal{K}(\reg(X))$ for some  $k\in \{0,...,m\}$. Then $$\dim(\ker(\Delta_k^{\mathcal{F}}))\leq \dim(I^{q_r}H^k(X,\mathbb{R}))$$ for each $r\in [2,\infty)$. If $X$ is also normal then $$\dim(\ker(\Delta_k^{\mathcal{F}}))\leq \dim(H^k(X,\mathbb{R})).$$
\end{prop}

\begin{proof}
According to \cite[Cor. 18]{FB} we know that $\Delta_k^{\mathcal{F}}:L^2\Omega^k(\reg(X),g)\rightarrow L^2\Omega^k(\reg(X),g)$ is a Fredholm operator on its domain endowed with the graph norm. In particular $\dim(\ker(\Delta_k^{\mathcal{F}}))<\infty$ and $\im(\Delta_k^{\mathcal{F}})=\im((d_{k-1}+d^t_k)_{2,\max})$ is closed in $L^2\Omega^k(\reg(X),g)$.  Now consider any $ \omega\in \ker(\Delta_k^{\mathcal{F}})$. Using \eqref{Fred} we know that $\omega\in \ker((d_{k}+d^t_{k-1})_{2,\min})$ and so $\omega\in \ker(d_{k,\min,2})$. Since $\Delta_k$ is an elliptic operator we know that $\omega$ is smooth and by Prop. \ref{tozzi} we know that $\omega \in L^r\Omega^k(\reg(X),g)$ for any $2\leq r\leq \infty$. So we can deduce that $\omega\in \ker(d_{k,\max,r})$ for any $2\leq r\leq \infty$. On the other hand by the fact that $\omega\in \ker(\Delta_k^{\mathcal{F}})$ we know that if $\omega\in \im((d_{k-1}+d^t_k)_{2,\max})$ then $\omega=0$. In particular if $\omega\in \im(d_{k-1,2,\max})$ then $\omega=0$, in that $\im(d_{k-1,2,\max})\subset \im((d_{k-1}+d^t_k)_{2,\max})$. Using the fact that $\vol_g(\reg(X))<\infty$ and arguing as in the proof of Th. \ref{tozziti} we get that if $\omega\in \im(d_{k-1,r,\max})$ then $\omega=0$, for each $2\leq r\leq \infty$. This leads us to conclude that if $\omega\neq 0$ then $\omega$ induces a non trivial class in  $H^k_{r,\max}(\reg(X),g)$ for any $2\leq r\leq \infty$ and the corresponding map $\ker(\Delta_k^{\mathcal{F}})\rightarrow H^k_{r,\max}(\reg(X),g)$ is  injective. Now the first inequality follows by Th. \ref{LpInt} while the second one follows by Th. \ref{mischiozzo} and the fact that $q_r=t$ for $r$ sufficiently big.
\end{proof}

Some of the most interesting applications concern compact smoothly Thom-Mather-Witt stratified pseudomanifolds.

\begin{teo}
\label{09/07}
Let $X$ be a compact and oriented smoothly  Thom-Mather-Witt stratified pseudomanifolds of dimension $m>2$. Let $g$ be  an iterated conic metric on $\reg(X)$ such that $d+d^t:L^2\Omega^{\bullet}(\reg(X),g)\rightarrow L^2\Omega^{\bullet}(\reg(X),g)$ is essentially self-adjoint.  We have the following properties:
\begin{enumerate}
\item If $\ell_k^-\in \mathcal{K}(\reg(X))$ for some $k\in \{0,...,m\}$ then Th. \ref{tozziti} holds true for $(\reg(X),g)$ and $k$. As a consequence $$\dim(I^{\underline{m}}H^k(X,\mathbb{R}))\leq \dim(I^{q_r}H^k(X,\mathbb{R}))$$ for any $2\leq r<\infty$. If $X$ is also normal then $$\dim(I^{\underline{m}}H^k(X,\mathbb{R}))\leq \dim(H^k(X,\mathbb{R})).$$
\item Let $2\leq r< \infty$ be arbitrarily fixed. If  $\ell_{k-1}^-,\ell_k^-\in \mathcal{K}(\reg(X))$ for some $k\in \{0,...,m\}$ then the map $$\gamma:H^k_{r,\max}(\reg(X),g)\rightarrow H^k_{2}(\reg(X),g)$$ is injective, see Th. \ref{LpL2}. Consequently $$\dim(I^{\underline{m}}H^k(X,\mathbb{R}))= \dim(I^{q_r}H^k(X,\mathbb{R}))$$ for any $2\leq r<\infty$. If $X$ is also normal then $$\dim(I^{\underline{m}}H^k(X,\mathbb{R}))= \dim(H^k(X,\mathbb{R})).$$
\item Assume that there exists $r\in (2,\infty)$ such that each singular stratum $Y\subset X$ satisfies $\mathrm{cod}(Y)\geq r$ if $\mathrm{depth}(Y)=1$ whereas $\mathrm{cod}(Y)> r$ if $\mathrm{depth}(Y)>1$. If $\ell_k^-,\ell_{k+1}^-\in \mathcal{K}(\reg(X))$ for some $k\in \{0,...,m\}$ then the $L^z$-Stokes theorem holds true on $L^z\Omega^k(M,g)$ for any $z\in [2,r]$. 
\end{enumerate}
\end{teo}

\begin{proof}
Since $d+d^t:L^2\Omega^{\bullet}(\reg(X),g)\rightarrow L^2\Omega^{\bullet}(\reg(X),g)$ is essentially self-adjoint we know that for each $j=0,...,m$, $\Delta_{j,\mathrm{abs}}:L^2\Omega^j(\reg(X),g)\rightarrow L^2\Omega^j(\reg(X),g)$  equals $\Delta_{j}^{\mathcal{F}}:L^2\Omega^j(\reg(X),g)\rightarrow L^2\Omega^j(\reg(X),g)$, see Prop. \ref{oneto}. Moreover Cor. \ref{closed} tells us that $\dim(H^j_{r,\max}(\reg(X),g))<\infty$ and  $\im(d_{j,\max,r})$ is closed in $L^r\Omega^{j+1}(\reg(X),g)$ for each $j=0,...,m$ and $2\leq r<\infty$. Now it is clear that Th. \ref{tozziti} holds true for $(\reg(X),g)$ and $k$. Consequently the inequality  $\dim(I^{\underline{m}}H^k(X,\mathbb{R}))\leq \dim(I^{q_r}H^k(X,\mathbb{R}))$ for any $2\leq r<\infty$ follows immediately by Th. \ref{tozziti} and Th. \ref{LpInt}. Finally if $X$ is also normal, by taking $r$ sufficiently big and using Th. \ref{mischiozzo}, we have  $\dim(I^{\underline{m}}H^k(X,\mathbb{R}))\leq \dim(I^{q_r}H^k(X,\mathbb{R}))$ $= \dim(I^{t}H^k(X,\mathbb{R}))$ $= \dim(H^k(X,\mathbb{R}))$. Concerning the second point we are in position to apply Th. \ref{LpL2} as $(\reg(X),g)$ satisfies all the corresponding assumptions. Thus $\gamma:H^k_{r,\max}(\reg(X),g)\rightarrow H^k_{2}(\reg(X),g)$ is injective. This, combined with the first point and Th. \ref{LpInt}, implies immediately that $\dim(I^{\underline{m}}H^k(X,\mathbb{R}))=$ $ \dim(I^{q_r}H^k(X,\mathbb{R}))$ for any $2\leq r<\infty$. In particular when $X$ is also normal, by taking $r$ sufficiently big, we  have  $\dim(I^{\underline{m}}H^k(X,\mathbb{R}))= \dim(H^k(X,\mathbb{R})).$ Finally by the assumptions made in the third point we know that $(\reg(X),g)$ is $r$-parabolic, see \cite[Th. 3.4]{BeGu}. Now the conclusion follows by Th. \ref{LpStokes}.
\end{proof}

\noindent We observe that it is already well known that the $L^2$-Stokes theorem holds true on any compact smoothly Witt-Thom-Mather stratified pseudomanifold without any assumption on the curvature. See \cite{ALMP} and \cite{JC}. The novelty of the third point of the above theorem lies in the cases $(2,r]$. Now we point out that when $k=1$ the first two points of the above theorem simplify considerably:

\begin{cor}
\label{ricci}
In the setting of Th. \ref{09/07}. Let $2\leq r<\infty$ be arbitrarily fixed. If $\mathrm{ric}^-\in \mathcal{K}(\reg(X))$ then Th. \ref{tozziti} and Th. \ref{LpL2} hold true for $(\reg(X),g)$ and $k=1$. Consequently 
the map $$\gamma:H^1_{r,\max}(\reg(X),g)\rightarrow H^1_{2}(\reg(X),g)$$ is an isomorphism and  $$\dim(I^{\underline{m}}H^1(X,\mathbb{R}))= \dim(I^{q_r}H^1(X,\mathbb{R}))$$ for any $2\leq r<\infty$. If $X$ is also normal then $$\dim(I^{\underline{m}}H^1(X,\mathbb{R}))= \dim(H^1(X,\mathbb{R})).$$
\end{cor}

\begin{proof}
This follows by Th. \ref{09/07} in the case $k=1$ and the corresponding Weitzenb\"ock formula.
\end{proof}

Let us now describe a class of examples to which Cor. \ref{ricci} applies. Let $\overline{M}$ be a compact and oriented manifold with boundary $N:=\partial\overline{M}$ and interior $M:=\overline{M}\setminus \partial\overline{M}$. Let us assume what follows:
there exist compact manifolds $B$ and $F$ such that $N$ is the total space of a fiber bundle $p:N\rightarrow B$ with fibers $F$  and structure group $G$ induced by a principal $G$-bundle $\tilde{p}:P\rightarrow B$ endowed with an arbitrarily fixed principal connection $\theta$. Moreover we require that $F$ is either even dimensional or $H^{f/2}(F,\mathbb{R})=0$, with $f:=\dim F$, that $F$ carries a Riemannian metric $h$ such that $\mathrm{sec}_h\geq 1$ and that $G$ acts on $(F,h)$ by isometries. 
From now on we fix a $G$-atlas $\mathcal{A}_G$ on the bundle $(N,B,p,F)$. Let $T_VN$ be the vertical tangent bundle of $p:N\rightarrow B$. Note that for any $(U,\varphi_U)$, $(V,\varphi_V)\in \mathcal{A}_G$ with $U\cap V\neq \emptyset$ we have $(\varphi_U^*h)|_{p^{-1}(U\cap V)}=(\varphi_V^*h)|_{p^{-1}(U\cap V)}$. Let thus $\xi\in C^{\infty}(N,(T_VN)^*\otimes (T_VN)^*)$ be the metric on $T_VN$ defined by patching together the local metrics $\varphi_U^*h$. Let now $g_B$ be an arbitrary Riemannian metric on $B$ and let $W$ be the horizontal subbundle of $TN$ induced by $\theta$. In this way we have $TN\cong W\oplus T_VN$. Let $\rho$ be the Riemannian metric on $N$ given by $\rho:=p^*g_B+\sigma$ with $\sigma$ the section of $T^*N\otimes T^*N$ induced by $TN\cong W\oplus T_VN$ and $\xi$. As a next step let us consider the fiber bundle $\pi: (0,1)\times N \rightarrow B$ with $\pi(r,x)=p(x)$ for each $x\in N$ and trivializations $\{(U,\phi_U)\}$ with $\phi_U:\pi^{-1}(U)\rightarrow (0,1)\times U\times F$ given by $\phi_U:= (\mathrm{Id},\varphi_U)$ and $(U,\varphi)\in \mathcal{A}_G$. We endow $(0,1)\times N$ with the metric $\tau:=\pi^*g_B+dr^2+r^2\tilde{\sigma}$, with $\tilde{\sigma}=\frak{p}^*\sigma$ and $\frak{p}:(0,1)\times N\rightarrow N$ the projection on the first factor. Note that both $p:(N,\rho)\rightarrow (B,g_B)$ and $\pi:((0,1)\times N,\tau)\rightarrow (B,g_B)$ become Riemannian submersions with totally geodesic fibers, see \cite[Th. 3.5]{Vilms}. Moreover in the former case  $(p^{-1}(b),\rho|_{p^{-1}(b)})$ is isometric to $(F,h)$ for each $b\in B$ while in the latter case $(\pi^{-1}(b),\tau|_{p^{-1}(b)})$ is isometric to $((0,1)\times F,dr^2+r^2h)$ for each $b\in B$. Clearly the construction of the metric $\tau$ can be performed by replacing $h$ with $ch$ where $c$ is any arbitrarily fixed positive constant. Finally let $g$ be any Riemannian metric on $M$ such that, for some collar neighborhood $\psi:\overline{U}\rightarrow [0,1)\times N$, we have $\psi^*\tau=g|_U$, with $U:=\overline{U}\setminus \partial\overline{U}$. Let us now continue by introducing the space $X$ defined as $X:=\overline{M}/\sim$ with $x\sim y$ if and only if $x,y\in N$ and $\pi(x)=\pi(y).$ It is easy to check that $X$ becomes a smooth Thom-Mather-Witt stratified pseudomanifold of depth one with one singular stratum given by $B$ and corresponding link $F$. In particular the collar neighborhood $\psi:\overline{U}\rightarrow [0,1)\times N$ and the fiber bundle $\pi:(0,1)\times N\rightarrow B$ induce on $X$ an open neighborhood $T_B$ of $B$ and a map $\pi_B:T_B\rightarrow B$ which are a retraction and a locally trivial fibration over $B$ with fiber $C(F)$, respectively. Note that $\reg(X)$ is diffeomorphic to $M$ and therefore $\reg(X)$ inherits the metric $g$ from $M$. It is easy to check that $(\reg(X),g)$ satisfies Def. \ref{iter}. Summarizing $X$ is a Thom-Mather-Witt stratified pseudomanifold of depth one whose regular part $\reg(X)$ is endowed with a conic metric $g$ (sometimes in this setting it is also called incomplete edge metric or wedge metric).

\begin{prop}
\label{exampleRicci}
Let $X$ and $g$ be as describe above. Then:  
\begin{enumerate}
\item If we scale the metric $h$ by a suitable constant $c$  then the corresponding Hodge-de Rham operator $d+d^t:L^2\Omega^{\bullet}(\reg(X),g)\rightarrow L^2\Omega^{\bullet}(\reg(X),g)$ is essentially self-adjoint;
\item For any positive $c$ we have  $\mathrm{ric}^-\in \mathcal{K}(\reg(X))$;
\end{enumerate}
Therefore for $k=1$, $c$ suitably fixed  and any  $2\leq r< \infty,  $Th. \ref{tozziti} and Th. \ref{LpL2} hold true. Consequently 
the map $$\gamma:H^1_{r,\max}(\reg(X),g)\rightarrow H^1_{2}(\reg(X),g)$$ is an isomorphism and  $$\dim(I^{\underline{m}}H^1(X,\mathbb{R}))= \dim(I^{q_r}H^1(X,\mathbb{R}))$$ for any $2\leq r<\infty$. If $X$ is also normal then $$\dim(I^{\underline{m}}H^1(X,\mathbb{R}))= \dim(H^1(X,\mathbb{R})).$$
\end{prop}
\begin{proof}
By  \cite[Th. 6.1, Prop. 5.4]{ALMP}, see also \cite{JB} and \cite{PPBV}, we know that by rescaling with a suitable positive constant $c$ the corresponding Hodge-de Rham operator $d+d^t:L^2\Omega^{\bullet}(\reg(X),g)\rightarrow L^2\Omega^{\bullet}(\reg(X),g)$ is essentially self-adjoint. In order to prove the second point it is enough to show that there exists a constant $\gamma$ such that $\mathrm{sec}_g\geq \gamma$ on $\reg(T_B)$. Since $(\reg(T_B),g|_{\reg(T_B)})$ is isometric to $((0,1)\times N,\pi^*g_B+dr^2+cr^2\tilde{\sigma})$ and $\pi:((0,1)\times N,\pi^*g_B+dr^2+cr^2\tilde{\sigma})\rightarrow (B,g_B)$ is a Riemannian submersion with totally geodesic fibers  we can use \cite{BON} to analyze the sectional curvatures of $(\reg(T_B),g|_{\reg(T_B)})$. More precisely from the third formula in \cite[Cor. 1]{BON} we can deduce that there exists a constant $\delta$ such that for any $q\in (0,1)\times N$ and any pair of horizontal tangent vectors $u,v\in T_{q}((0,1)\times N)$ with $|u|_{\tau}=|v|_{\tau}=1$ and $\tau(u,v)=0$, we have $\mathrm{sec}_\tau(u,v) \geq \gamma$. The second formula in \cite[Cor. 1]{BON} tells us that $\mathrm{sec}_{\tau}(u,v)\geq 0$ whenever $u,v\in T_q((0,1)\times N)$ are a vertical and a horizontal tangent vector, respectively. We are left to examine the case of two vertical tangent vectors. We know that $(\pi^{-1}(b),\tau|_{p^{-1}(b)})$ is isometric to $((0,1)\times F,dr^2+cr^2h)$ for each $b\in B$. The sectional curvatures of $((0,1)\times F,dr^2+cr^2h)$ are nonnegative, as a consequence of the fact that  $\mathrm{sec}_{ch}=\mathrm{sec}_h\geq 1$ and  the calculations carried out in \cite[p. 27-28]{spinconic} or \cite[App. A]{peterli}. Since  $\pi:((0,1)\times N,\tau)\rightarrow (B,g_B)$ is a Riemannian submersion with totally geodesic fibers we can use the first formula in \cite[Cor. 1]{BON} to conclude that $\mathrm{sec}_{\tau}\geq 0$ on vertical tangent vectors.  Summarizing we showed that $(\reg(X),g)$ has sectional curvatures bounded from below and thus, in particular, we have  $\mathrm{ric}^-\in \mathcal{K}(\reg(X))$. The remaining statements of this proposition are now an immediate consequence of Th. \ref{09/07} and Cor. \ref{ricci}.
\end{proof}

Furthermore we point out that other examples of stratified pseudomanifolds carrying an (iterated) conic metric with Ricci curvature bounded from below are discussed in \cite{IlaMo} and \cite{Vertman}.  We have now the next:

\begin{cor}
Let $X$ be a compact and oriented smoothly Thom-Mather-Witt stratified pseudomanifold of dimension $m$. Assume that $\dim(I^{\underline{m}}H^1(X,\mathbb{R}))\neq \dim(I^{q_r}H^1(X,\mathbb{R})$ for some  $2\leq r<\infty$. Then there is no iterated conic metric $g$ on $\reg(X)$ such that both $d+d^t:L^2\Omega^{\bullet}(\reg(X),g)\rightarrow L^2\Omega^{\bullet}(\reg(X),g)$ is essentially self-adjoint and $\mathrm{ric}^-\in \mathcal{K}(\reg(X))$. In particular there is no iterated conic metric $g$ on $\reg(X)$ such that both $d+d^t:L^2\Omega^{\bullet}(\reg(X),g)\rightarrow L^2\Omega^{\bullet}(\reg(X),g)$ is essentially self-adjoint and $\mathrm{Ric}\geq c$ for some $c\in \mathbb{R}$. 
\end{cor}

\begin{proof}
The first assertion  follows immediately by Cor. \ref{ricci}. The second one follows by the first one and the fact that if $\mathrm{Ric}\geq c$ for some $c\in \mathbb{R}$ then $\mathrm{ric}\in \mathcal{K}(\reg(X))$.
\end{proof}

We continue with the following vanishing result.

\begin{prop}
\label{vanishing}
In the setting of Th. \ref{09/07}. The following properties hold true:
\begin{enumerate}
\item If $L_k\geq 0$ and $L_{k,p}>0$ for some $k\in \{1,...,m-1\}$ and $p\in \reg(X)$ then $I^{\underline{m}}H^k(X,\mathbb{R})=\{0\}$.
\item If $\ell^-_{k-1}\in \mathcal{K}(\reg(X))$,\ $L_k\geq 0$ and $L_{k,p}>0$ for some  $k\in \{1,...,m-1\}$ and $p\in \reg(X)$ then $I^{q_r}H^k(X,\mathbb{R})=\{0\}$ for each $2\leq r <\infty$.
\end{enumerate}
\end{prop}
\begin{proof}
The proof of the first point follows the well known strategy used in the classical Bochner-type vanishing theorems. Since the manifold is incomplete, integration by part necessitates a justification. In virtue of \eqref{isoabs} and Cor. \ref{closed} it is enough to show that $\mathcal{H}^k_{2,\mathrm{abs}}(\reg(X),g)=\{0\}$. Moreover, since we are in position to use Prop. \ref{oneto}, the previous equality boils down to showing that $\ker(\Delta_k^{\mathcal{F}})=\{0\}$. Let $\omega$ be any $k$-forms with $\omega\in \ker(\Delta_k^{\mathcal{F}})$. As shown in \cite[Prop. 3.5]{FrB} we know that $\omega\in \mathcal{D}(\nabla_{\min})$ with $$\nabla_{\min}:L^2\Omega^k(\reg(X),g)\rightarrow L^2(\reg(X), T^*\reg(X)\otimes \Lambda^k(\reg(X)),g)$$ the minimal extension of the connection $\nabla: \Omega^k_c(\reg(X))\rightarrow C^{\infty}_c(\reg(X),T^*\reg(X)\otimes \Lambda^k(\reg(X)))$ induced by the Levi-Civita connection. Moreover, by elliptic regularity and Th. \ref{tozzi}, we know that $\omega\in \Omega^{k}(\reg(X))\cap L^{\infty}\Omega^k(\reg(X),g)$, that is $\omega$ is smooth and bounded. 
 Consider now a sequence $\{\phi_n\}_{n\in \mathbb{N}}\subset \mathrm{Lip}_c(M,g)$ that makes $(\reg(X),g)$ parabolic. We have 
\begin{align}
& \nonumber 
\langle \Delta_k^{\mathcal{F}}\omega ,\phi_n^2 \omega\rangle_{L^2\Omega^k(\reg(X),g)}= \langle \nabla^t(\nabla \omega)+L_k\omega ,\phi_n^2 \omega\rangle_{L^2\Omega^k(\reg(X),g)}=\\
& \nonumber \langle \nabla \omega ,\nabla(\phi_n^2 \omega)\rangle_{L^2(\reg(X),T^*\reg(X)\otimes \Lambda^k(\reg(X)),g)}+ \langle L_k(\phi_n\omega) ,\phi_n \omega\rangle_{L^2\Omega^k(\reg(X),g)}=\\
\nonumber & \langle \nabla \omega ,\phi_n^2\nabla \omega\rangle_{L^2(\reg(X),T^*\reg(X)\otimes \Lambda^k(\reg(X)),g)}+ \langle \nabla \omega ,2\phi_n(d_{0}\phi_n)\otimes \omega\rangle_{L^2(\reg(X),T^*\reg(X)\otimes \Lambda^k(\reg(X)),g)}+\\
& \nonumber \langle L_k(\phi_n\omega) ,\phi_n \omega\rangle_{L^2\Omega^k(\reg(X),g)}.
\end{align}

Using the Lebesgue dominate convergence theorem, the Cauchy-Schwartz inequality and the inequality $$\|(d_{0}\phi_n)\otimes \omega\|_{L^2(\reg(X),T^*\reg(X)\otimes \Lambda^k(\reg(X)),g)}\leq \|d_{0}\phi_n\|_{L^2\Omega^1(\reg(X),g)}\|\omega\|_{L^{\infty}\Omega^k(\reg(X),g)}$$ we have
\begin{align}
\nonumber & \lim_{n\rightarrow \infty} \langle \Delta_k^{\mathcal{F}}\omega ,\phi_n^2 \omega\rangle_{L^2\Omega^k(\reg(X),g)}
=\langle \Delta_k^{\mathcal{F}}\omega ,\omega\rangle_{L^2\Omega^k(\reg(X),g)}=0,\\
& \nonumber \lim_{k\rightarrow \infty} \langle \nabla \omega ,\phi_n^2\nabla \omega\rangle_{L^2(\reg(X),T^*\reg(X)\otimes \Lambda^k(\reg(X)),g)}=\langle \nabla \omega ,\nabla \omega\rangle_{L^2(\reg(X),T^*\reg(X)\otimes \Lambda^k(\reg(X)),g)},\\
\nonumber & \lim_{n\rightarrow \infty}\langle \nabla \omega ,2\phi_n(d_{0}\phi_n)\otimes \omega\rangle_{L^2(\reg(X),T^*\reg(X)\otimes \Lambda^k(\reg(X)),g)}=0.
\end{align}

In this way we know that $\lim_{n\rightarrow \infty}\langle L_k(\phi_n\omega) ,\phi_n \omega\rangle_{L^2\Omega^k(\reg(X),g)}$ exists and $$0=\|\nabla \omega\|^2_{L^2(\reg(X),T^*\reg(X)\otimes \Lambda^k(\reg(X)),g)}+\lim_{n\rightarrow \infty}\langle L_k(\phi_n\omega) ,\phi_n \omega\rangle_{L^2\Omega^k(\reg(X),g)}.$$ Since $\langle L_k(\phi_n\omega) ,\phi_n \omega\rangle_{L^2\Omega^k(\reg(X),g)}\geq 0$ for each $n\in \mathbb{N}$ we can immediately conclude that $$\lim_{n\rightarrow \infty}\langle L_k(\phi_n\omega) ,\phi_n \omega\rangle_{L^2\Omega^k(\reg(X),g)}= 0\ \quad\quad \mathrm{and}\quad\quad  \nabla \omega=0$$ and thus, thanks to the Kato inequality, we know that $|\omega|_g$ is constant. Finally let us examine the term $\langle L_k(\phi_n\omega) ,\phi_n \omega\rangle_{L^2\Omega^k(\reg(X),g)}$. We can write it  as $$\langle L_k(\phi_n\omega) ,\phi_n \omega\rangle_{L^2\Omega^k(\reg(X),g)}=\int_{\reg(X)}g(L_k\phi_n\omega,\phi_n\omega)\dvol_g=\|g(L_k\phi_n\omega,\phi_n\omega)\|_{L^1(\reg(X),g)}$$ where in the last equality we used that  $g(L_k\phi_n\omega,\phi_n\omega)\geq 0$. Hence can deduce that 
$$\lim_{n\rightarrow \infty}\|g(L_k\phi_n\omega,\phi_n\omega)\|_{L^1(\reg(X),g)}=0$$ which in turn implies $$\lim_{n\rightarrow \infty}g(L_k\phi_n\omega,\phi_n\omega)=0$$ pointwise almost everywhere on $\reg(X)$. On the other hand it is clear that $$\lim_{n\rightarrow \infty}g(L_k\phi_n\omega,\phi_n\omega)=g(L_k\omega,\omega)$$ pointwise almost everywhere on $\reg(X)$. Thus we get that $g(L_k\omega,\omega)=0$ a.e. on $\reg(X)$. As $L_k\geq 0$ and $L_{k,p}>0$, the form $\omega$ has to vanish in a neighborhood of $p$ and finally, using the fact that $|\omega|_g$ is constant, we can conclude that $\omega=0$. This concludes the proof of the first point. The second point follows immediately by the first one and Th. \ref{09/07}.
\end{proof}

\begin{cor}
In the setting of Th. \ref{09/07}. If $\mathrm{Ric}\geq 0$ and $\mathrm{Ric}_p>0$ for some point $p\in \reg(X)$ then $I^{\underline{m}}H^1(X,\mathbb{R})=I^{q_r}H^1(X,\mathbb{R})=\{0\}$ for any $2\leq r<\infty$.
\end{cor}
\begin{proof}
This follows immediately by Cor. \ref{ricci} and Prop. \ref{vanishing}.
\end{proof}

Now we focus on a special case of compact smoothly Thom-Mather stratified pseudomanifolds. Let $\overline{M}$ a compact  manifold with boundary. Let us denote with $\partial \overline{M}$ and $M$ the boundary and the interior of $\overline{M}$, respectively. Let $g$ be any smooth symmetric section of $T^*\overline{M}\otimes T^*\overline{M}$ that restricts to a Riemannian metric on $M$ and such   that  there exists an open neighborhood $U$ of $\overline{M}$ and a diffeomorphism  $\psi:  [0,1)\times \partial\overline{M}\rightarrow U$ such that $$\psi^*(g|_U)=dx^2+x^2h(x)$$ with $h(x)$ a family of Riemannian metrics on $\partial\overline{M}$ that depends smoothly on $x$ up to $0$. Finally let $X$ be the quotient space defined by $\overline{M}/\sim$ with $p\sim q$ if and only if $p$ and $q\in \partial\overline{M}$. It is immediate to check that $X$ becomes a compact smoothly Thom-Mather stratified pseudomanifold  with only one isolated singularity whose regular part is diffeomorphic to $M$. With a little abuse of notation we still denote with $g$ the Riemannian metric that $\reg(X)$ inherits from $(M,g)$. Clearly Def. \ref{iter} is satisfied by $(\reg(X),g)$.  As we will see, there are at least two important reasons for an in-depth examination of this class of singular spaces: first we  can drop the Witt assumption in many cases and moreover the topological implications arising from the curvature  are sometimes given in terms of the usual singular homology.  

\begin{prop}
\label{conic}
Let $X$ and $g$ be as above with $\dim(X)=m>2$ and let $\nu:=m/2$ if $m$ is even while $\nu:=(m-1)/2$ if $m$ is odd. We have the following properties:
\begin{enumerate}
\item If $\ell_k^-\in \mathcal{K}(\reg(X))$ for some  $k\in \{0,...,m\}$, with $k\neq \nu$ if $m$ is even whereas $k\neq \nu,\nu+1$ if $m$ is odd, then Th. \ref{tozziti} holds true for $(\reg(X),g)$ and $k$. Consequently $$\dim(I^{\overline{m}}H_{m-k}(X,\mathbb{R}))\leq \dim(I^{p_r}H_{m-k}(X,\mathbb{R}))$$ for any $2\leq r<\infty$. Furthermore if  $1<k<\nu$ we have $$\dim(H_{m-k}(X),\mathbb{R})\leq \dim(H_{m-k}(\reg(X),\mathbb{R}))$$ whereas if $k=1$ we have $$\dim(H_{m-1}(X),\mathbb{R})=\dim(\im(H_{m-1}(\reg(X),\mathbb{R})\rightarrow H_{m-1}(X,\mathbb{R}))).$$
\item Let $2\leq r< \infty$ be arbitrarily fixed. If $\ell_{k-1}^-,\ell_k^-\in \mathcal{K}(\reg(X))$ for some $k\in \{0,...,m\}$, with $k-1,k\neq \nu$ if $m$ is even whereas $k-1,k\notin \{ \nu,\nu+1\}$ if $m$ is odd, then the map $$\gamma:H^k_{r,\max}(\reg(X),g)\rightarrow H^k_{2}(\reg(X),g)$$ is an isomorphism, see Th. \ref{LpL2} and Cor. \ref{IsoL2Lq}. Consequently $$\dim(I^{\overline{m}}H_{m-k}(X,\mathbb{R}))= \dim(I^{p_r}H_{m-k}(X,\mathbb{R})).$$  In particular if $1<k<\nu$ then $$\dim(H_{m-k}(\reg(X),\mathbb{R}))= \dim(H_{m-k}(X,\mathbb{R})).$$ 
\item If $\ell_k^-,\ell_{k+1}^-\in \mathcal{K}(\reg(X))$ for some  $k\in \{0,...,m\}$, with $k,k+1\neq \nu$ if $m$ is even whereas $k,k+1\notin \{\nu,\nu+1\}$ if $m$ is odd,  then the $L^z$-Stokes theorem holds true on $L^z\Omega^k(M,g)$ for any $z\in [2,m]$. 
\end{enumerate}
\end{prop} 

\begin{proof}
Thanks to \cite[Th. 3.7, Th. 3.8]{BLE} we know that $\Delta_j^{\mathcal{F}}$ equals $\Delta_{j,\mathrm{abs}}$ for each $j\in \{0,1,...,m\}$, with $j\neq \nu$ if $m$ is even while  $j\notin \{\nu,\nu+1\}$ if $m$ is odd. Therefore Th. \ref{tozziti} holds true for $(\reg(X),g)$ and $k$ and hence, arguing as in Th. \ref{09/07}, we can conclude that $$\dim(I^{\overline{m}}H_{m-k}(X,\mathbb{R}))\leq \dim(I^{p_r}H_{m-k}(X,\mathbb{R}))$$ for any $2\leq r<\infty$. Assume now that  $1<k<\nu$. As showed in \cite[Example 4.1.12]{Bana} we have 
\begin{equation}
\label{banagl}
I^pH_i(X,\mathbb{R})\cong \left\{
\begin{array}{lll}
H_i(\reg(X),\mathbb{R}) & i<m-1-p(m)\\ 
\im(H_i(\reg(X),\mathbb{R})\rightarrow H_i(X,\mathbb{R})) & i=m-1-p(m)\\
H_i(X,\mathbb{R}) & i>m-1-p(m)
\end{array}
\right.
\end{equation}
 with $p$ any perversity. If we choose $r$  big enough we have $p_r(m)=0$ and thus we get $I^{p_r}H_i(X,\mathbb{R})\cong H_i(\reg(X),\mathbb{R})$ with $0\leq i<m-1$ and $I^{p_r}H_{m-1}(X,\mathbb{R})\cong \im(H_{m-1}(\reg(X),\mathbb{R})\rightarrow H_{m-1}(X,\mathbb{R}))$. On the other hand we have seen that $\dim(I^{\overline{m}}H_{m-k}(X,\mathbb{R}))\leq \dim(I^{p_r}H_{m-k}(X,\mathbb{R}))$  and according to \eqref{banagl} we have $I^{\overline{m}}H_{m-i}(X,\mathbb{R})\cong H_{m-i}(X,\mathbb{R})$ if  $m-i>m-1-[\frac{m-1}{2}]$ that is,  if $i<1+[\frac{m-1}{2}]$. Since $1+[\frac{m-1}{2}]=m/2=\nu$ when $m$ is even whereas  $1+[\frac{m-1}{2}]=\frac{m+1}{2}=\nu+1$ when $m$ is odd we can conclude that 
if $1<k<\nu$ it holds $$\dim(H_{m-k}(X,\mathbb{R}))\leq \dim(H_{m-k}(\reg(X),\mathbb{R}))$$ whereas if $k=1$ then $$\dim(H_{m-1}(X,\mathbb{R}))=\dim(\im(H_{m-1}(\reg(X),\mathbb{R})\rightarrow H_{m-1}(X,\mathbb{R}))).$$
Concerning the second point we are in position to apply Th. \ref{LpL2} and Cor. \ref{IsoL2Lq}. Therefore $\gamma:H^k_{r,\max}(\reg(X),g)\rightarrow H^k_{2}(\reg(X),g)$ is an isomorphism and $\dim(I^{\overline{m}}H_{m-k}(X,\mathbb{R}))= \dim(I^{p_r}H_{m-k}(X,\mathbb{R}))$. In particular if  $1<k<\nu$, by choosing $r$ sufficiently big and applying \eqref{banagl}, the previous equality becomes $\dim(H_{m-k}(\reg(X),\mathbb{R}))= \dim(H_{m-k}(X,\mathbb{R}))$. Finally we tackle the last point. Thanks to \cite[Cor. 3.7]{BeGu} we know that $(\reg(X),g)$ is $z$-parabolic for any $z\in [1,m]$. Now it is clear that all the assumptions of Th. \ref{LpStokes} are fulfilled and so the  $L^z$-Stokes theorem holds true on $L^z\Omega^k(M,g)$ for any $z\in [2,m]$. 
\end{proof}
Also in this case we remark that the $L^2$-Stokes theorem was already known without curvature assumptions, see \cite{BLE}. The novelty, likewise Th. \ref{09/07}, lies in the cases $(2,m]$. 

\begin{cor}
\label{vanishingconic}
In the setting of Prop. \ref{conic}. The following properties hold true:
\begin{enumerate}
\item If $L_k\geq 0$ and $L_{k,p}>0$ for some $p\in \reg(X)$ and $k\in \{1,...,m-1\}$, with $k\neq \nu$ if $m$ is even whereas $k\neq \nu,\nu+1$ if $m$ is odd, then $$I^{\overline{m}}H_{m-k}(X,\mathbb{R})=\{0\}.$$
\item If $\ell^-_{k-1}\in \mathcal{K}(\reg(X))$,\ $L_k\geq 0$ and $L_{k,p}>0$ for some  $p\in \reg(X)$ and $k\in \{1,...,m-1\}$, with $k\neq \nu$ if $m$ is even whereas $k-1,k\notin\{\nu,\nu+1\}$ if $m$ is odd, then $$I^{p_r}H_{m-k}(X,\mathbb{R})=\{0\}$$ for each $2\leq r <\infty$. In particular if $1<k<\nu$ then $$H_{m-k}(\reg(X),\mathbb{R})=H_{m-k}(X,\mathbb{R})=\{0\}.$$
\item If $\mathrm{Ric}\geq 0$ and $\mathrm{Ric}_p>0$ for some point $p\in \reg(X)$ then $$I^{\overline{m}}H_{m-1}(X,\mathbb{R})=I^{p_r}H_{m-1}(X,\mathbb{R})=\im(H_{m-1}(\reg(X),\mathbb{R})\rightarrow H_{m-1}(X,\mathbb{R}))=\{0\}$$ for any $2\leq r<\infty$.
\end{enumerate}
\end{cor}
\begin{proof}
This follows immediately by Prop. \ref{vanishing}, Prop. \ref{conic} and \eqref{banagl}.
\end{proof}

Clearly the above results can be used to exhibit examples of compact smoothly Thom-Mather stratified pseudomanifolds with isolated singularities that do not carry any conic metric having the negative part of $L_k$ in the Kato class. To this aim we have the next

\begin{prop}
\label{noRiccibb}
Let $N$ be a compact and oriented manifold with $\dim(H^1(N,\mathbb{R}))>1$.  Let $\overline{M}:=N\times [0,1]$ and let $X$ be the space obtained by collapsing the boundary of $\overline{M}$ to a point, see Prop. \ref{conic}. Then there is no conic metric $g$ on $\reg(X)$ such that $\mathrm{ric^-}\in \mathcal{K}(\reg(X))$. In particular there is no conic metric  on $\reg(X)$ with Ricci curvature bounded from below.
\end{prop}

\begin{proof}
Clearly $\reg(X)\cong M\times (0,1)$ and $\dim(H^1(M\times (0,1),\mathbb{R}))=\dim(H^1(M,\mathbb{R}))>1$. By using Lefschetz  duality we have $H^1(M\times (0,1),\mathbb{R})\cong H_{m-1}(\overline{M},\partial\overline{M},\mathbb{R})$ and finally, as explained for instance in  \cite[Rmk. 4.4.2]{KiWoo}, we have $H_{m-1}(\overline{M},\partial\overline{M},\mathbb{R})\cong H_{m-1}(X,\mathbb{R})$. Eventually we showed that $\dim(H_{m-1}(X,\mathbb{R}))>1$. On the other hand we have $\dim(H_{m-1}(\reg(X),\mathbb{R}))=\dim(H_c^1(M\times (0,1),\mathbb{R}))=1$, see \cite[Prop. 4.7]{Bott}, and so, by virtue of Prop. \ref{conic}, we can conclude that there is no conic metric $g$ on $\reg(X)$ such that $\mathrm{ric^-}\in \mathcal{K}(\reg(X))$ since $\dim(H_{m-1}(X,\mathbb{R}))>\dim(H_{m-1}(\reg(X),\mathbb{R}))$. The remaining assertions follow immediately.
\end{proof}

\begin{cor}
Let $N_1$, $N_2$ be  compact and oriented manifolds with $\dim(H^1(N_j,\mathbb{R})>0$, $j=1,2$. Let $N:=N_1\times N_2$. Then, with the same notation as in Prop. \ref{noRiccibb}, there is no conic metric  $g$ on $\reg(X)$ with $\mathrm{ric}^-\in \mathcal{K}(\reg(X))$. In particular there is no conic metric  on $\reg(X)$ with Ricci curvature bounded from below.
\end{cor}
\begin{proof}
This follows immediately by Prop. \ref{noRiccibb} as $\dim(H^1(N,\mathbb{R}))>1$.
\end{proof}

In view of the next corollary we recall that $\Sigma N$, the suspension of $N$, is defined as $\Sigma N:= N\times [0,1]/\sim$
with $p\sim q$ if and only if either $p,q\in N\times \{0\}$ or $p,q\in N\times \{1\}$.

\begin{cor}
Let $N$ be a compact and oriented manifold with $\dim(H^1(N,\mathbb{R})>1$.  Then there is no conic metric $g$ on $\reg(\Sigma N)$ with $\mathrm{ric}^-\in \mathcal{K}(\reg(\Sigma N))$.  In particular there is no conic metric  on $\reg(\Sigma N)$ with Ricci curvature bounded from below.
\end{cor}

\begin{proof}
It is easy to note that $\Sigma N$ is the normalization of $X$, with $X$ defined as in Prop. \ref{conic}. Since  $I^pH_k(X,\mathbb{R})\cong I^pH_k(\Sigma,\mathbb{R})$, see  \cite[Th. 6.6.6]{Bana} or \cite[pag. 151]{GoMac}, the conclusion follows immediately from Prop. \ref{conic} and Prop. \ref{noRiccibb}.
\end{proof}

\subsection{Applications to complex projective varities}
The last part of this paper contains some applications to complex projective varieties. For the corresponding definitions we refer to \cite{GHa}. Let us consider a complex projective variety $V\subset \mathbb{C}\mathbb{P}^n$ of complex dimension $v$ and such that $\dim(\sing(V))=0$. Let us denote with $h$ the K\"ahler metric on $\reg(V)$ induced by $g$, the Fubini-Study metric on $\mathbb{C}\mathbb{P}^n$. We have the following:

\begin{teo}
\label{projvar}
Let $V$ and $h$ be as above  with $v>1$. Assume that $\ell_k^-\in \mathcal{K}(\reg(V))$ for some $k\in \{0,...,2v\}$, $k\notin \{v\pm1, v\}$. Then
\begin{equation}
\label{cotto}
\dim(I^{\underline{m}}H^k(V,\mathbb{R}))\leq \dim(I^{t}H^k(V,\mathbb{R})).
\end{equation}
If $V$ is normal then 
\begin{equation}
\label{ricotto}
\dim(I^{\underline{m}}H^k(V,\mathbb{R}))\leq \dim(H^k(V,\mathbb{R})).
\end{equation}
 If  $1<k<v-1$ then $$ \dim(H_{2v-k}(V,\mathbb{R}))\leq \dim(H_{2v-k}(\reg(V),\mathbb{R}))$$ whereas if $k=1$ we have $$\dim(H_{2v-1}(V),\mathbb{R})=\dim(\im(H_{2v-1}(\reg(V),\mathbb{R})\rightarrow H_{2v-1}(V,\mathbb{R}))).$$
Finally if we assume  that $L_k\geq c$ and  $L_{k,p}\geq c$ for some $p\in \reg(V)$, then $$I^{\underline{m}}H^k(V,\mathbb{R})=\{0\}.$$
\end{teo}

\begin{proof}
It is well know that $(\reg(V),h)$ is an incomplete K\"ahler manifold of finite volume. For instance this can be easily  deduced using a resolution of singularities $\pi:M\rightarrow V$. Moreover we know that the Sobolev inequality \eqref{imagine} holds true for $(\reg(V),h)$, see \cite[Eq. (5.5)]{LT}, and $\Delta_j^{\mathcal{F}}:L^2\Omega^j(\reg(V),h)\rightarrow L^2\Omega^j(\reg(V),h)$ equals $\Delta_{j,\mathrm{abs}}:L^2\Omega^j(\reg(V),h)\rightarrow L^2\Omega^j(\reg(V),h)$ when $j\neq \{v,v\pm1\}$, see \cite[Th. 1.2]{GL}. Furthermore thanks to \cite{TOh} we know that $H^j_{2,\max}(\reg(V),h)\cong I^{\underline{m}}H^j(V,\mathbb{R})$ for each $j=0,...,2v$ which in turn implies that $\dim(H^j_{2,\max}(\reg(V),h))<\infty$ and $\im(d_{j,\max,2})$ is closed in $L^2\Omega^{j+1}(\reg(V),h)$ for each $j$. Hence thanks to Th. \ref{tozziti}, Prop. \ref{oneto} and the above remarks we know that $\alpha_{\infty}: H^k_2(\reg(V),h)\rightarrow \im(H^k_{s,\infty}(\reg(V),h)\rightarrow\overline{H}^k_{\infty,\max}(\reg(V),h))$ is injective. Now using the isomorphism $H^j_{s,\infty}(\reg(V),h)\cong I^{t}H^j(V,\mathbb{R})$ for each $j=0,...2v$, see \cite[Th. 1.2.2]{Val}, we can conclude that $$\dim(I^{\underline{m}}H^k(V,\mathbb{R}))\leq \dim(I^{t}H^k(V,\mathbb{R}))$$
and $$\dim(I^{\underline{m}}H^k(V,\mathbb{R}))\leq \dim(H^k(V,\mathbb{R}))$$ provided $V$ is normal. This shows the first two statements above. Concerning the third and the fourth one we know that \eqref{banagl} holds true for $V$ as well, see \cite[\S\ 6.1]{GoMac}. In this way we have $H^j_{2,\max}(\reg(V),g)\cong H_{2v-j}(V,\mathbb{R})$ for $2v-j>2v-1-[\frac{2v-1}{2}]$, that is $j<v$,  and  $H^{j}_{s,\infty}(\reg(V),h)\cong I^tH^j(V,\mathbb{R})\cong I^0H_{2v-j}(V,\mathbb{R})\cong H_{2v-j}(\reg(V),\mathbb{R})$ when $1<j\leq 2v$ while  $H^{1}_{s,\infty}(\reg(V),h)\cong I^0H_{2v-1}(V,\mathbb{R})\cong \im(H_{2v-1}(\reg(V),\mathbb{R})\rightarrow H_{2v-j}(V,\mathbb{R})$. Hence, keeping in mind that \eqref{cotto} holds true for $k\notin \{v\pm1,v\}$, we can conclude that $$\dim(H_{2v-k}(V,\mathbb{R}))\leq \dim(H_{2v-k}(\reg(V),\mathbb{R}))$$ if $1<k<v-1$ whereas if $k=1$  $$\dim(H_{2v-1}(V,\mathbb{R}))=\dim(\im(H_{2v-1}(\reg(V),\mathbb{R})\rightarrow H_{2v-1}(V,\mathbb{R}))).$$ Finally the last assertion follows using the first assertion of Prop. \ref{vanishing}.
\end{proof}

\begin{cor}
\label{ricciproj}
Let $V$ and $h$ be as above  with $v>1$. The following properties hold true
\begin{enumerate}
\item If $\mathrm{ric}^-\in \mathcal{K}(\reg(V))$ then $$\dim(I^{\underline{m}}H^1(V,\mathbb{R}))\leq \dim(I^{t}H^1(V,\mathbb{R})).$$ In particular if $V$ is normal then $$\dim(I^{\underline{m}}H^1(V,\mathbb{R}))\leq \dim(H^1(V,\mathbb{R})).$$
\item If $\mathrm{Ric}\geq 0$ and  $\mathrm{Ric}_p>0$ for some $p\in \reg(V)$, then $$I^{\underline{m}}H^1(V,\mathbb{R})=I^0H_{2v-1}(V,\mathbb{R})=\im(H_{2v-1}(\reg(V),\mathbb{R})\rightarrow H_{2v-1}(V,\mathbb{R}))=\{0\}.$$
\end{enumerate}
\end{cor}

\begin{proof}
This follows immediately by Prop. \ref{vanishing}, Prop. \ref{projvar} and \eqref{banagl}.
\end{proof}

\begin{cor}
Let $V$ be as above. If $\dim(H_{2v-1}(V,\mathbb{R}))>\dim(\im(H_{2v-1}(\reg(V),\mathbb{R})\rightarrow \im(H_{2v-1}(V,\mathbb{R}))$ then there is no Riemannian metric $\rho$ on $\mathbb{C}\mathbb{P}^n$ with $\mathrm{ric}_{\sigma}^-\in \mathcal{K}(\reg(V))$, where $\sigma$ denotes the Riemannian metric induced by $\rho$ on $\reg(V)$. In particular there is no Riemannian metric on $\mathbb{C}\mathbb{P}^n$ that induces a Riemannian metric on $\reg(V)$ with Ricci curvature bounded from below.
\end{cor}
\begin{proof}
This follows immediately by Prop. \ref{projvar}, the fact that $\rho$ is quasi-isometric to the Fubini-Study metric and that the $L^2$-cohomology is stable through quasi-isometries.
\end{proof}

Using the above results we construct now some examples of singular projective varieties with only isolated singularities and such that $\mathrm{ric}^-\notin \mathcal{K}(\reg(V))$. Let $N\subset \mathbb{C}\mathbb{P}^{s}$ be a smooth projective variety of complex dimension $n>1$. Assume that $\dim(H^1(N,\mathbb{R}))\geq 1$. Let $V$ be the projective cone of $N$. We recall that $V$ is the Zariski closure of $C(N)$ in $\mathbb{C}\mathbb{P}^{s+1}$ and $C(N)$, the affine cone over $N$, is defined as $C(N)=\{0,...,0\}\cup \theta^{-1}(N)$ with $\theta:\mathbb{C}^{s+1}\setminus \{0,...,0\}\rightarrow \mathbb{C}\mathbb{P}^s$ the map that sends the point with affine coordinates $(a_0,...,a_s)$ to the point with homogeneous coordinates $[a_0:...:a_s]$, see\cite[p.12]{RoHa}. Then it is easy to check that $V$ has only one isolated singularity $\sing(V)=\{[0:...:0:1]\}$. 
Let $h$ be the metric on $\reg(V)$ induced by the Fubini-Study metric of $\mathbb{C}\mathbb{P}^{s+1}$. We want to show that $\mathrm{ric}^-\notin \mathcal{K}(\reg(V))$. To this aim we use \eqref{cotto} and we prove that $\dim(I^{\underline{m}}H^1(V,\mathbb{R}))>\dim(I^tH^1(V,\mathbb{R})).$ Concerning $I^{\underline{m}}H^1(V,\mathbb{R})$ we have $\dim(I^{\underline{m}}H^1(V,\mathbb{R}))=\dim(H^1(\reg(V),\mathbb{R}))$ $=\dim(H^{1}(N,\mathbb{R}))\geq 1$ as $N$ and $\reg(V)$ are homotopically equivalent. For the other cohomology group, using \cite[\S\ 6.1]{GoMac} we have 
\begin{align}
& \nonumber \dim(I^tH^1(V,\mathbb{R}))=\dim(\im(H^1_c(\reg(V),\mathbb{R})\rightarrow H^1(\reg(V),\mathbb{R})))\leq\dim(H_c^1(\reg(V),\mathbb{R}))=\\
& \nonumber \dim(H_{2n+1}(\reg(V),\mathbb{R}))=\dim(H_{2n+1}(N,\mathbb{R}))=0.
\end{align}
 Note that the equalities in the second line above follow  by the fact that $N$ and $\reg(V)$ are homotopically equivalent and that the real dimension of $N$ is $2n$. We can thus conclude that  $\mathrm{ric}^-\notin \mathcal{K}(\reg(V))$ as desired.

\end{document}